\documentclass{article}
\usepackage{amsmath,amssymb}
\usepackage{amsthm}
\usepackage{natbib}
\usepackage[linesnumbered,ruled]{algorithm2e}
\usepackage{float}

\usepackage{graphicx}
\usepackage{caption}
\usepackage{subcaption}
\usepackage{centernot}

\usepackage{geometry}
\usepackage[makeroom]{cancel}
\title{Rank $2r$ iterative least squares: efficient recovery of ill-conditioned low rank matrices from few entries}

\author{Jonathan Bauch\footnotemark[1]\thanks{Faculty of Mathematics and Computer Science, Weizmann Institute of Science \newline (jonathan.bauch@weizmann.ac.il, boaz.nadler@weizmann.ac.il, pizilber@gmail.com)}
\and Boaz Nadler\footnotemark[1]
\and Pini Zilber\footnotemark[1]}

\DeclareMathOperator*{\argmin}{argmin}
\DeclareMathOperator{\Vectorize}{Vec}

\DeclareMathOperator{\ColNorm}{ColNorm}

\DeclareMathOperator{\Span}{Span}

\newtheorem{theorem}{Theorem}[]
\newtheorem{lemma}{Lemma}[]
\newtheorem{corollary}{Corollary}[]
\newtheorem{remark}{Remark}[]

\begin{document}

\maketitle

\begin{abstract}
We present a new, simple and computationally efficient iterative method 
for low rank matrix completion. Our method is inspired by the class of factorization-type
iterative algorithms, but substantially differs from them in the way the problem is cast.  
Precisely, given a target rank $r$, instead of optimizing on the manifold of rank $r$ matrices, we allow our interim estimated matrix to have a specific over-parametrized
rank $2r$ structure. 
Our algorithm, denoted \texttt{R2RILS}, for rank $2r$ iterative least squares, thus has low memory requirements, and at each iteration it solves a computationally cheap   
sparse least-squares problem. We motivate our algorithm by its theoretical analysis
for the simplified case of a rank-1 matrix.  
Empirically, \texttt{R2RILS} is able to recover ill conditioned low rank matrices from very few observations -- near the information limit,
and it is stable to additive noise. 
\end{abstract}

\section{Introduction}
Consider the following matrix
completion problem, whereby the goal is to estimate an unknown $m\times n$ matrix $X_0$ given only few of its
entries, possibly corrupted by noise. For this problem to be well posed,  following many previous works, 
we assume that the underlying matrix $X_0$ is exactly of rank $r$, with $r\ll \min(m,n)$
and that it satisfies incoherence conditions as detailed below. For simplicity we further assume that the rank $r$ is a-priori known. 
Formally, let  $\Omega\subset [m]\times [n]$ be the subset of observed indices, 
and $X$ the matrix with observed entries in $\Omega$ and zero in its complement $\Omega^c$. 
For any matrix $A$, denote $\|A\|_{F(\Omega)}^2=\sum_{(i,j)\in\Omega} A_{ij}^2$, with a similar definition for $\|A\|_{F(\Omega^c)}$. Then, the problem is 
\begin{equation}
        \label{eq:LRMC}
\min_Z  \|Z-X\|_{F(\Omega)} \quad \mbox{subject to } {\mbox{ rank}(Z)\leq r }.
\end{equation}

Problem (\ref{eq:LRMC})
is intimately related to matrix factorization and principal component analysis with missing data, 
which date back to the 1970's \citep{wiberg1976}. 
Such problems appear in a variety of applications, including 
collaborative filtering, global positioning in wireless sensor networks, system identification and structure from motion,
see \citep{buchanan2005damped,candes2010matrix,davenport2016overview} and references therein. 
In some applications, such as global positioning and structure from motion, the underlying matrix is exactly low rank, though the measurements may be corrupted by noise. In other applications, such as collaborative filtering, the underlying matrix is only approximately low rank. 
For the vast literature on low rank matrix completion,  
see the reviews \citep{candes2010matrix, chi2019matrix, chi2019nonconvex,davenport2016overview}.

In this work we focus on recovery of matrices with an {\em exact} rank $r$. 
For an $m\times n$ matrix $X_0$ of rank $r$, we denote its non-zero singular values by 
$\sigma_1\geq \sigma_2\geq\ldots\geq \sigma_r>0$,  
and its condition number by $\sigma_1 / \sigma_r$. 
On the theoretical front, several works studied perfect recovery of a rank $r$ matrix $X_0$ from only few entries. 
A key property allowing the recovery of \(X_{0}\) is  {\em incoherence} of its row and column subspaces \citep{candes2009exact,candes2010power,gross2011recovering}. The ability to recover a low rank matrix is also related to rigidity theory \citep{singer2010uniqueness}. 
Regarding the set \(\Omega\), a necessary condition for well-posedness of the matrix completion problem (\ref{eq:LRMC}) is that $|\Omega|\geq r\cdot(m+n-r)$ which is the number of free parameters for a rank $r$ matrix.  
Another necessary condition is that the 
set $\Omega$ contains at least $r$ entries in each row and column \citep{pimentel2016characterization}.  
When the entries of $\Omega$ are sampled uniformly at random, as few as $O(r (m+n) \mbox{polylog}(m+n))$ entries suffice to exactly recover an incoherent rank $r$ matrix  $X_0$. 
For a given set $\Omega$, we denote its oversampling ratio by 
$\rho=\frac{|\Omega|}{r(m+n-r)}$.
In general, the closer $\rho$ is to the value one, the harder the matrix completion task is.

On the algorithmic side, most methods for low rank matrix completion can be assigned to one of two classes. One class consists of algorithms which optimize over the full $m\times n$ matrix, whereas the second class consists 
of methods that explicitly enforce the rank $r$ constraint in (\ref{eq:LRMC}). 
Several methods in the first class replace the rank constraint by a low-rank inducing penalty  $g(Z)$. In the absence of noise, 
this leads to the following optimization problem, 
\begin{equation}
        \label{eq:rho_Z_equal}
\min_Z g(Z) \quad \mbox{such that } Z_{ij}=X_{ij} \ \ \forall (i,j)\in\Omega.
\end{equation}
When the observed entries are noisy a popular objective is 
\begin{equation}
        \label{eq:rho_Z_LS}
\min_Z \|Z-X\|_{F(\Omega)}^2 + \lambda g(Z),
\end{equation}
where the parameter $\lambda$ is often tuned via some cross-validation procedure. 

Perhaps the most popular penalty is the nuclear norm, also known as the trace norm,
and given by $g(Z)=\sum_i \sigma_i(Z)$, where $\sigma_i(Z)$ are the singular values of $Z$ 
\citep{fazel2001rank}. 
As this penalty is convex, both (\ref{eq:rho_Z_equal}) and (\ref{eq:rho_Z_LS}) lead to convex semi-definite programs, which may be solved in polynomial time.  However,
even for modest-sized matrices with hundreds of rows and columns, standard solvers have prohibitively long runtimes. Hence, several works proposed fast optimization methods, 
see
\citep{avron2012efficient,  cai2010singular,fornasier2011low,ji2009accelerated,ma2011fixed,mazumder2010spectral,rennie2005fast, toh2010accelerated} and references therein. 
On the theoretical side, under suitable conditions and with a sufficient number of observed entries, nuclear norm minimization provably recovers, with high probability, the underlying low rank matrix and is also stable to additive noise in the observed entries \citep{candes2010matrix,candes2009exact,candes2010power,
gross2011recovering,recht2011simpler}.

As noted by \citep{tanner2013normalized}, the nuclear norm penalty fails to recover  
low rank matrices at low oversampling ratios.  
Recovery in such data-poor settings is possible using non-convex matrix penalties such as the Schatten $p$-norm with $p<1$  \citep{marjanovic2012l_q,kummerle2018harmonic}. 
However, optimizing such norms may be computationally demanding.  
Figure \ref{fig:HM_comparison} compares the runtime and recovery error of 
\texttt{HM-IRLS} optimizing the Schatten $p$-norm with $p=1/2$ \citep{kummerle2018harmonic} and of our proposed method \texttt{R2RILS}, as a function of matrix size $m$ with $n=m+100$. For example, for a rank-10 matrix of size $700\times 800$, 
\texttt{HM-IRLS}  required over 5000 seconds, whereas \texttt{R2RILS} 
took about 20 seconds. 

The second class consists of iterative methods that strictly enforce the rank $r$ constraint of Eq. (\ref{eq:LRMC}). 
This includes hard thresholding methods that keep 
at each iteration only the top $r$ singular values and vectors \citep{tanner2013normalized,blanchard2015cgiht,kyrillidis2014matrix}. 
More related to our work are methods based on a rank $r$ factorization $Z = U V^\top$ where $U\in\mathbb{R}^{m\times r}$ and $V\in\mathbb{R}^{n\times r}$. 
Problem (\ref{eq:LRMC}) now reads 
\begin{equation}
                \label{eq:min_factorization}
        \min_{U,V} \| U V^\top - X\|_{F(\Omega)}.
\end{equation}
Whereas Eq. (\ref{eq:rho_Z_LS}) involves $mn$ optimization variables, 
problem (\ref{eq:min_factorization}) involves only
$r(m+n)$ variables, making it scalable to large matrices.

One approach to optimize Eq. (\ref{eq:min_factorization}) is by {\em alternating minimization} \citep{haldar2009rank,keshavan2010matrix,tanner2016low,wen2012solving}. 
Each iteration first solves a least squares problem for $V$, keeping the column space estimate $U$ fixed. 
Next, keeping the new $V$ fixed, it optimizes over $U$. 
%
Under suitable conditions, 
alternating minimization provably recover the low rank matrix, 
with high probability \citep{hardt2014understanding,  
jain2015fast,jain2013low,keshavan2010matrix,
sun2016guaranteed}.


\begin{figure}[t]
\centering
\begin{subfigure}{.5\textwidth}
  \centering
\includegraphics[width=0.85\textwidth]{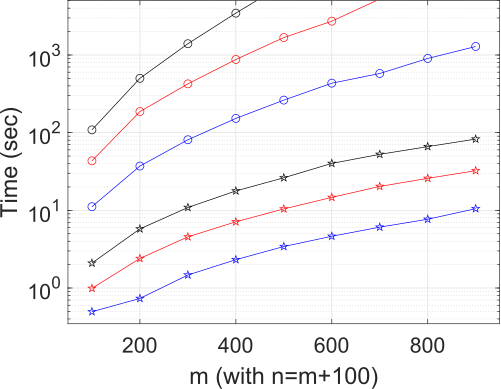}
  \caption{Execution time}
\label{fig:timing_time_comparison}
\end{subfigure}%
\begin{subfigure}{.5\textwidth}
  \centering
\includegraphics[width=0.85\textwidth]{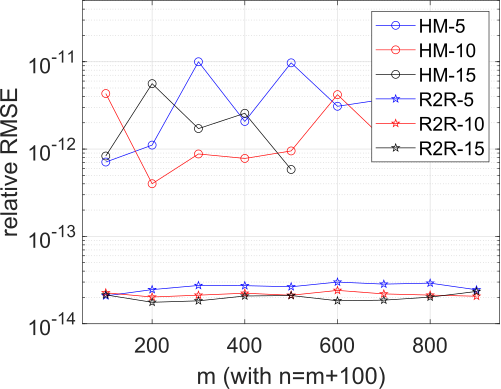}  
  \caption{Recovery RMSE}
  \label{fig:timing_error_comparison}
\end{subfigure}
\caption{Comparison of \texttt{HM-IRLS} \citep{kummerle2018harmonic} and   \texttt{R2RILS} for completion of rank $r$ matrices of size $m\times (m+100)$ as a function of $m$, at an oversampling ratio of $\rho=2.5$. For each $r\in\{5,10,15\}$, all non-zero singular values were one. (a) runtime; (b)  
Relative RMSE on the unobserved entries, Eq. (\ref{eq:rel-RMSE}). Note that the $y$-axis in both graphs is logarithmic. Results of  \texttt{HM-IRLS} at large values of $m$  are not shown, as we capped individual runs to 3 hours. }
        \label{fig:HM_comparison}
\end{figure}

Another iterative approach to optimize (\ref{eq:min_factorization}) was
	proposed in the 1970's by Wiberg \citep{wiberg1976}, and later became popular in the computer vision community.
	Given a guess $V^t$, let $U(V^t)$ be the closed form solution to the alternating step of minimizing (\ref{eq:min_factorization}) with respect to $U$. Wiberg's method
	writes the new $V$ as $V=V^{t}+\Delta V$, and performs a Newton approximation to the functional $\|U(V)V^\top - X\|_{F(\Omega)}$. This yields a degenerate least squares problem for $\Delta V$. 
	In \citep{okatani2011efficient}, a damped Wiberg method
	was proposed with improved convergence and speed. 

Yet a different approach to minimize Eq. (\ref{eq:min_factorization}) is via gradient-based Riemannian manifold optimization \citep{vandereycken2013low,boumal2015low,mishra2014r3mc,mishra2014fixed, ngo2012scaled}.  
For recovery guarantees of such methods, see \citep{wei2016guarantees}.
For scalability to large matrices, \citep{balzano2010online} devised a stochastic gradient descent approach, called \texttt{GROUSE}, whereas
\citep{recht2013parallel} devised a parallel scheme called \texttt{JELLYFISH}. 
Finally, uncertainty quantification in noisy matrix completion was recently addressed in \citep{chen2019inference}. 

While factorization-based methods are 
fast and scalable, 
they have two limitations: 
(i) several of them fail to recover even mildly ill-conditioned low rank matrices and (ii) they require relatively large oversampling ratios to succeed. 
In applications, the underlying matrices may have a significant spread in their singular values, and clearly the ability to recover a low rank matrix from even a factor of two fewer observations may be of great importance. 

Let us illustrate these two issues. 
With a full description in Section \ref{sec:numerical_results}, Figure \ref{fig:test_1}  
shows that with a condition number of 1, several popular algorithms recover the low rank matrix. However, as shown in Figure \ref{fig:test_10}, once the condition number is 10, many algorithms 
either require a high oversampling ratio, or
 fail to recover the matrix to high accuracy.  
In contrast, \texttt{R2RILS} recovers the low rank matrices 
from fewer entries and is less sensitive to the ill conditioning.

\begin{figure}[t]
	\centering
	\begin{subfigure}{.5\textwidth}
		\centering
		\includegraphics[width=0.85\linewidth]{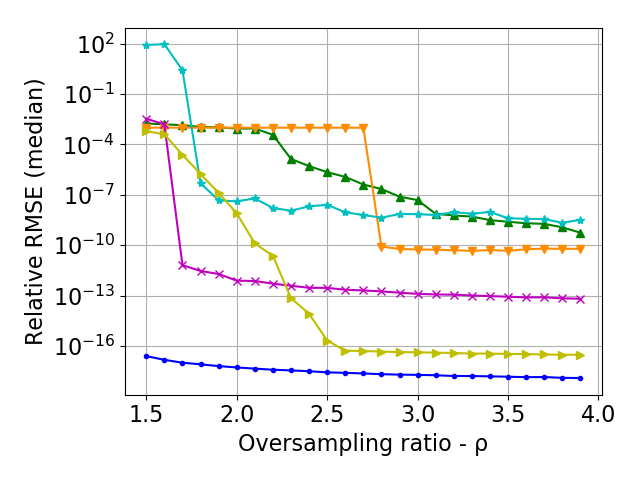}
		\caption{Median of Normalized RMSE}
		\label{fig:cond1_RMSE}
	\end{subfigure}%
	\begin{subfigure}{.5\textwidth}
		\centering
		\includegraphics[width=0.85\linewidth]{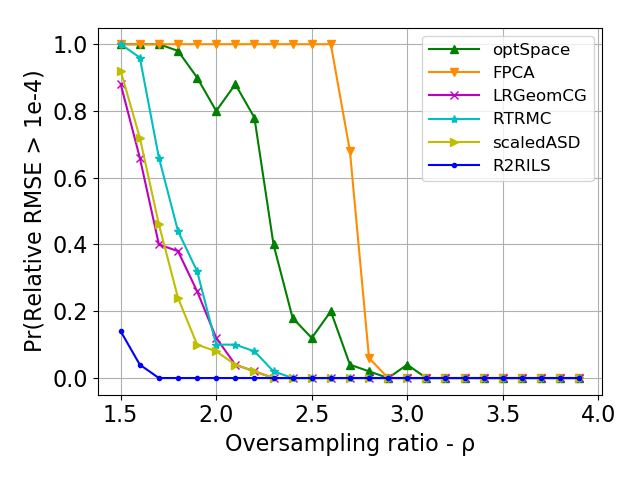}
		\caption{Failure probability}
		\label{fig:cond1_prob}
	\end{subfigure}
	\caption{Comparison of several matrix completion algorithms with well-conditioned matrices of size
		$1000\times 1000$ and rank $r=5$ as a function of the oversampling ratio $\rho$. (a) median of \texttt{rel-RMSE}, Eq. (\ref{eq:rel-RMSE}); (b) failure probability, defined as \texttt{rel-RMSE} $>10^{-4}$. Each point on the two graphs corresponds to  50 independent realizations.}
	\label{fig:test_1}
\end{figure}

\paragraph{Our Contributions}
 In this paper, we present \texttt{R2RILS}, a novel iterative method for matrix completion that is simple to implement, computationally efficient, scalable 
 and performs well both with few
observations, ill conditioned matrices and noise.
Described in Section \ref{sec:algorithm}, \texttt{R2RILS} is inspired
by factorization
algorithms. However, it substantially differs from them, since 
given a target rank $r$, it does not directly optimize Eq. (\ref{eq:min_factorization}).
Instead, we allow our interim matrix to have a specific over-parametrized rank $2r$ structure.
Optimizing over this rank $2r$ matrix yields a least squares problem. At each iteration, \texttt{R2RILS} 
thus simultaneously 
optimizes the column and row subspaces. A key step 
is then an {\em averaging} of these new and current estimates. 

Section \ref{sec:numerical_results} presents an empirical evaluation of \texttt{R2RILS}. First, we consider 
rank $r$ incoherent matrices, with entries observed uniformly at random. 
We show that in noise-free settings, \texttt{R2RILS}  
exactly completes matrices from fewer entries than several other low rank completion methods. We further show that \texttt{R2RILS} is robust to 
 ill-conditioning of the underlying matrix and to additive noise. 
 Next, we 
 consider a different type of ill-conditioning, namely power-law matrices which are much less incoherent. Finally, 
 we compare \texttt{R2RILS} to the damped Wiberg's algorithm
 on two datasets from computer vision, where the later method
 was shown to obtain state-of-the-art results. This comparison also highlights some limitations of our method.

 To provide insight and motivation for our approach, 
  in Section \ref{sec:theory} we study some of its theoretical properties, under the simplified setting of a rank-1 matrix, both
  with noise free observations as well as with observations corrupted by additive Gaussian noise.
  While beyond the scope of the current manuscript, we remark that the approach we present in this work can be extended to several other problems, including Poisson low rank matrix completion, one bit matrix completion, and low rank tensor completion. These will be described in future works.

\begin{figure}[t]
\centering
\begin{subfigure}{.5\textwidth}
  \centering
  \includegraphics[width=0.85\linewidth]{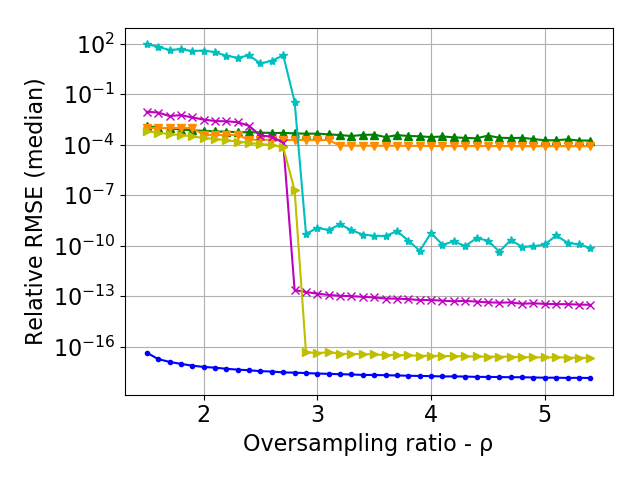}
  \caption{Median of Normalized RMSE}
  \label{fig:cond10_RMSE}
\end{subfigure}%
\begin{subfigure}{.5\textwidth}
  \centering
  \includegraphics[width=0.85\linewidth]{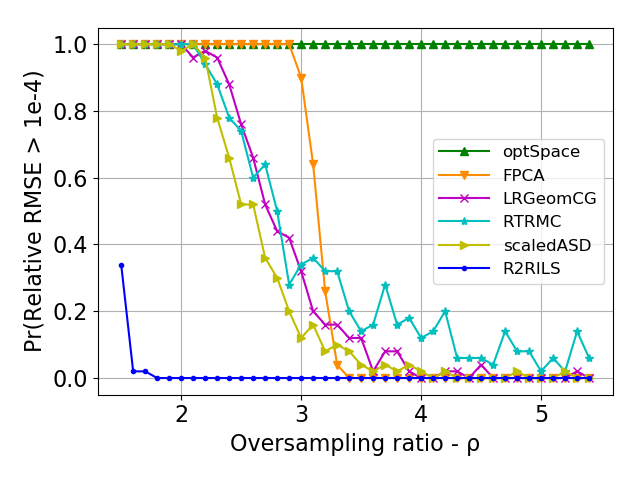}
  \caption{Failure probability}
  \label{fig:cond10_prob}
\end{subfigure}
\caption{
Similar comparison as in Fig.~\ref{fig:test_1}, but now the rank $5$ matrices have singular values $10,8,4,2,1$, resulting in a condition number of $10$.
}
        \label{fig:test_10}
\end{figure}

\section{The \texttt{R2RILS} Algorithm}
        \label{sec:algorithm}
While motivated by factorization methods, a key difference of \texttt{R2RILS} is that it does
not directly 
optimize the objective of  Eq. (\ref{eq:min_factorization}). Instead, 
\texttt{R2RILS} utilizes a specific {\em lifting} to the space of rank $2r$ matrices. 
Let $(U_t,V_t)$ be the estimates of the column and row spaces of the rank $r$ matrix $X_0$, at the start of iteration $t$. 
Consider the subspace of rank $2r$ matrices of the following specific form, 
with  $A \in \mathbb{R}^{m\times r}, B \in \mathbb{R}^{n\times r}$. 
\[
U_t B^{\top} + AV_{t}^{\top},
\]
Starting from an initial guess $(U_1,V_1)$
at $t=1$, \texttt{R2RILS} iterates the following two steps:  

\noindent
{\bf Step I.} Compute the minimal $\ell_2$-norm solution $\left(\tilde U_t, \tilde V_t\right)$ of the following least squares problem
\begin{equation} \label{eq:ILS_update_rule_step1}
\argmin_{A \in \mathbb{R}^{m\times r}, B \in \mathbb{R}^{n\times r}} \| U_{t}B^\top + AV_{t}^\top -  X  \|_{F(\Omega)} .
\end{equation}
{\bf Step II.} Update the row and column subspace estimates,
\begin{equation}
        \label{eq:ILS_update_rule_step2}
        \begin{split}
                U_{t+1} &= \ColNorm\left(U_{t} + \ColNorm\left(\tilde{U}_t\right) \right),\\
                V_{t+1} &= \ColNorm\left(V_{t} + \ColNorm\left(\tilde{V}_t\right) \right),
        \end{split}
\end{equation}
where $\ColNorm(A)$ normalizes all $r$ columns of the matrix $A$ 
to have unit norm. 

\begin{algorithm}[t]
    \caption{\texttt{R2RILS}}   \label{Alg:R2RILS}
    \SetKwInOut{Input}{Input}
    \SetKwInOut{Output}{Output}
    \SetKwRepeat{Do}{do}{while}

    \Input{
    $\Omega$ - the set of observed entries.\\ 
    $X$ - an $m\times n$ matrix with the observed values in $\Omega$ and zeros in $\Omega^c$, \\
    $r$ - the target rank \\
    $U_1$, $V_1$ - initial guess for the column and row rank-r subspaces \\
    $t_{\max}$ - maximal number of iterations
    }
    \Output{$\hat X$ - rank $r$ approximation of $X$}
    \For{ $t=1,\ldots,t_{\max}$}
    {
        Compute $(\tilde{U}_t, \tilde{V}_t)$, the minimal norm solution of $\argmin\limits_{A \in \mathbb{R}^{m \times r}, B \in \mathbb{R}^{n \times r}} \left\Vert    
        U_{t}B^{\top} + AV_{t}^{\top} - X 
        \right\Vert_{F(\Omega)}$\\
        
                $U_{t+1} = \ColNorm\left(U_{t} + \ColNorm\left(\tilde{U}_t\right) \right)$\\
                
                $V_{t+1} = \ColNorm\left(V_{t} + \ColNorm\left(\tilde{V}_t\right) \right)$        
    }
    \Return $\hat X =$ rank $r$ approximation of $(U_t\tilde{V}_t^{\top} + \tilde{U}_tV_t^{\top})$ at the iteration t which achieved the smallest squared error on the observed entries
\end{algorithm}

At each iteration we compute the rank $r$ projection of 
the rank-$2r$ matrix $U_t\tilde V_t^\top + \tilde U_t V_t^\top$
and its squared error on the observed entries.
The output of \texttt{R2RILS} is the rank $r$ approximation of $(U_t\tilde{V}_t^{\top} + \tilde{U}_tV_t^{\top})$ which minimizes the squared error on the observed entries. 
A pseudo-code of \texttt{R2RILS} appears in Algorithm \ref{Alg:R2RILS}. 
As we prove in Lemma  \ref{lemma:critical_point}, if \texttt{R2RILS} converges, then its limiting solution is rank $r$. 
Next, we provide intuition for \texttt{R2RILS} and discuss some of its differences from other factorization-based methods. 

\vspace{1em}
\noindent
{\bf Rank Deficiency.}
As discussed in Lemma \ref{lemma:rank_deficiency} below, Eq.~(\ref{eq:ILS_update_rule_step1}) is a rank deficient least squares problem and therefore does not have a unique solution.  
\texttt{R2RILS} takes the solution with smallest Euclidean norm, which is unique. 
%
We remark that the least squares linear system in Wiberg's method, although different, 
	has a similar rank deficiency, see \citep[Proposition 1]{okatani2007wiberg}. 
\vspace{1em}

\noindent
{\bf Simultaneous Row and Column Optimization.} 
In Eq.~(\ref{eq:ILS_update_rule_step1}) of Step I, \texttt{R2RILS} finds the best
approximation of $X$ by a linear combination  of the current row and column subspace estimates $(U_t,V_t)$, using weight matrices $(A,B)$. 
Thus, 
Eq.~(\ref{eq:ILS_update_rule_step1}) {\em simultaneously} optimizes both
the column and row subpaces, generating new estimates for them 
$(\tilde U_t,\tilde V_t)$.
This scheme is significantly different from alternating minimization methods, which 
at each step optimize only one of the row or column subspaces,   
keeping the other fixed. 
In contrast, \texttt{R2RILS} {\em decouples} the estimates, at the expense of lifting to a rank $2r$ intermediate solution, 
\begin{equation}
        \label{eq:X_t_2r}
\hat{X}_t = U_{t}\tilde{V}_{t}^{\top} + \tilde{U}_{t}V_t^{\top}.
\end{equation}
\noindent
{\bf Tangent Space.} Another prism to look at Eq.~(\ref{eq:ILS_update_rule_step1}) is through its connection to the tangent space of the manifold of rank $r$ matrices. Consider a rank $r$ matrix with column and row subspaces spanned by $U_t$, $V_t$, i.e. $Z = U_t M V_t^{\top}$ where $M \in \mathbb{R}^{r\times r}$ is invertible. Then the rank $2r$ matrix of Eq.~(\ref{eq:X_t_2r}) is the best approximation of $X$ in the {\em tangent space} of $Z$, in least squares sense.
\vspace{1em}

\noindent
{\bf Averaging Current and New Estimates.} 
Since $(\tilde{U_{t}}, \tilde{V_{t}})$ can be thought of as 
new estimates for the column and row spaces, it is tempting to consider the update  
$U_{t+1}= \tilde{U}_{t}, V_{t+1} = \tilde{V}_{t}$, followed by column normalization.
While this update may seem attractive, leading to a non increasing sequence of losses for the objective in Eq.~(\ref{eq:ILS_update_rule_step1}), it performs very poorly.  
Empirically, this update decreases the objective in Eq.~(\ref{eq:ILS_update_rule_step1}) extremely slowly, 
taking thousands of iterations to converge. 
Moreover, as illustrated in Figure \ref{fig:Naive},
the resulting sequence
$\{(U_t,V_t)\}_t$ alternates between two sets of poor estimates. 
The left panel shows the RMSE on the observed entries versus iteration number for \texttt{R2RILS}, and for 
a variant whose update is $U_{t+1}= \tilde{U}_{t}, V_{t+1} = \tilde{V}_{t}$,
followed by column normalization. Both methods were initialized by the SVD of $X$. While \texttt{R2RILS} converged in 6 iterations, the naive method failed to converge even after 500 iterations. The right panel presents a 2-d visualization of the last 50 vectors $U_{t,1}$ under the naive update, as projected into the first and second SVD components of the matrix containing the last 50 values of $U_{t}$ for $t=451:500$. This graph shows an oscillating behavior between two poor vectors. 
Remark~\ref{rem:optimal_weighting} explains this behavior, in the rank-one case.

\begin{figure}[t]
	\centering
		\includegraphics[width=0.42\linewidth]{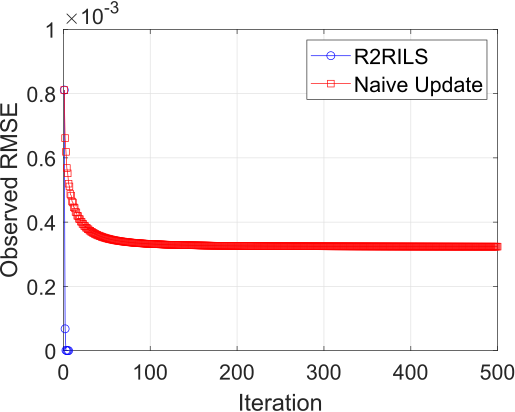}
		\includegraphics[width=0.43\linewidth]{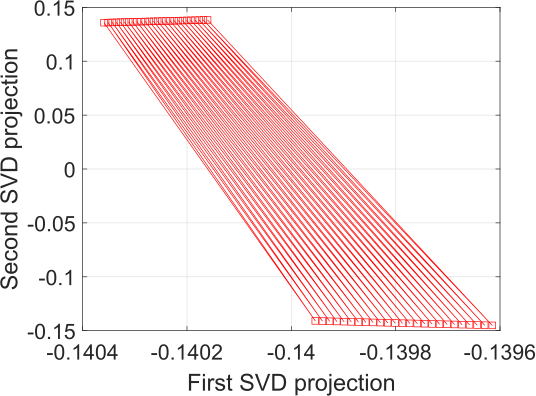}
	\caption{Comparison of \texttt{R2RILS} versus a variant with a naive update. The input is a $400\times 500$ matrix of rank $r=3$, with observed entries uniformly at random at an oversampling ratio $\rho=5$. (a) Observed RMSE as a function of iteration counter. 
    (b) Two dimensional visualization of the oscillating dynamics of $U_t$ under the naive update rule. }
	\label{fig:Naive}
\end{figure}

Nonetheless, thinking of $(\tilde{U}_t, \tilde{V}_t)$ as new estimates provides a useful perspective. In particular, if the error in $(\tilde{U}_t, \tilde{V}_t)$
is in a different direction than the error in the initial estimate $(U_t,V_t)$ or better yet, is approximately in the {\em opposite} direction, then the sensible operation
to perform is to {\em average} these two estimates.  This is indeed what \texttt{R2RILS} does in its second step.
In Section \ref{sec:theory} we show that in the rank-$1$ case, when the entire matrix is observed, the errors in $(\tilde{U}_{t}, \tilde{V}_{t})$ are indeed approximately in the opposite direction. Furthermore, we show that when the previous estimates are already close to the ground truth, the equal weighting of previous and current estimates is asymptotically optimal.
Specifically, this averaging cancels the leading order terms of the errors, and leads to quadratic convergence.
\vspace{1em}

\noindent
{\bf Non-Local Updates.}
Several Riemannian optimization methods, such as \texttt{LRGeomCG} \citep{vandereycken2013low} and \texttt{RTRMC} \citep{boumal2015low}, perform {\em local} optimization on the manifold of rank $r$ matrices, based on the gradient at the current solution. Our update rule is significantly different from these methods, since in Step I, we find the {\em global} minimizer of  Eq.~(\ref{eq:ILS_update_rule_step1}) in a specific rank $2r$ subspace. 
Given the averaging operation in the second step of \texttt{R2RILS}, 
its next estimate $(U_{t+1},V_{t+1})$ may be far from the current one $(U_t,V_t)$, in particular in the first few iterations.
\vspace{1em}

\noindent
{\bf Invariant Alternatives.} 
An intriguing property of \texttt{R2RILS} is that
its next estimate depends explicitly on the specific columns of $(U_t,V_t)$, and not only on the subspace they span. 
That is, \texttt{R2RILS} is not invariant to the representation of the current subspace, and 
does not treat $(U_t,V_t)$ as elements on the Grassmannian. 
%
It is possible to devise variants of \texttt{R2RILS} that are invariant to the subspace representation. 
%
One way to do so is to update $(U_{t+1}, V_{t+1})$ as the average subspace between $(U_t, V_t)$ and $(\tilde{U}_t, \tilde{V}_t)$, with respect to say the standard Stiefel geometry. 
Another invariant alternative is to take the best rank $r$ approximation of $\hat X_t$ from (\ref{eq:X_t_2r}) as the next estimate.
While these variants work well at high oversampling ratios, the simple column averaging 
Eq.~(\ref{eq:ILS_update_rule_step2}) outperforms them at low oversampling ratios. 
A theoretical understanding 
of this behavior
is an interesting topic for future research. 
\vspace{1em}

\noindent
{\bf Initialization.} Similar to other iterative algorithms,
	\texttt{R2RILS} requires an initial guess $U_1,V_1$. 
	When the observed entries are distributed uniformly at random, 
	a common choice is to initialize $U_1,V_1$ by the
	top $r$ left and right singular vectors of $X$. As described
	in \citep{keshavan2010matrix}, for a trimmed variant of $X$, these quantities are an accurate estimate of
	the left and right subspaces. In contrast, when the distribution of the observed entries is far from uniform, the rank-$r$ SVD of $X$ may be a poor initialization, and in various applications, one starts from a random guess, see \citep{okatani2011efficient}. 
Empirically, \texttt{R2RILS} performs well also from a random initialization, say with i.i.d.~$N(0,I)$ Gaussian vectors,
though it may require more iterations to converge. 
This suggests that the sequence $(U_t,V_t)$ computed by \texttt{R2RILS} is not attracted to poor local minima.
This finding is in accordance with \citep{ge2016matrix}, that rigorously proved lack of poor local minima for the matrix completion problem under suitable assumptions.

\vspace{1em}
\noindent
{\bf Early Stopping.} In the pseudo-code of Algorithm \ref{Alg:R2RILS}, the number of iterations is fixed at $t_{\max}$. In most of our simulations, we set $t_{\max}=300$, though 
in practice, the algorithm often converged in much fewer iterations. 
In our code we implemented several early stopping criteria.
Let $\hat X^{t}_r=\mbox{SVD}_r(U_t\tilde V^\top + \tilde U V_t^{\top} )$ denote the rank-$r$ approximation of $\hat X_t$ from Eq.~\eqref{eq:X_t_2r}, and let 
	$\mbox{RMSE}_{\text{obs}}^t = \|\hat X^{t}_r - X\|_{F(\Omega)}/\sqrt{|\Omega|}$ be the corresponding root mean squared error on the observed entries.
Then the first criterion, relevant only 
to the noise-free case, is
\begin{equation*}
\mbox{Stop if} \qquad \mbox{RMSE}_{\text{obs}}^t \leq \epsilon,
\end{equation*}
taking for instance $\epsilon \leq 10^{-15}$.
A second stopping criterion, valid also in the noisy case, is
\[
\mbox{Stop if} \qquad \frac{ \| \hat{X}_{t} - \hat{X}_{t-1}\|_{F} }{\sqrt{m n }}\leq \epsilon .
\]
Finally, a third stopping criterion, useful in particular with real data, is to set $\delta \ll 1$ and 
\[
\mbox{Stop if} \qquad | \mbox{RMSE}_{obs}^t- \mbox{RMSE}_{obs}^{t-1} | \leq \delta \cdot \mbox{RMSE}_{obs}^t
\]

\noindent
{\bf Computational complexity.} Our Matlab and Python implementations of \texttt{R2RILS}, available at the author's website\footnote{\texttt{www.wisdom.weizmann.ac.il/$\sim$nadler/Projects/R2RILS/}}, use standard linear algebra packages. Specifically, 
the minimal norm solution of Eq.~(\ref{eq:ILS_update_rule_step1}) is calculated by the LSQR iterative algorithm \citep{paige1982lsqr}.
The cost per iteration of LSQR is discussed in \citep[Section 7.7]{paige1982lsqr}.
In our case, it is dominated by 
the computation of $U_{t}B^\top + AV_{t}^\top$ at all entries in $\Omega$, whose complexity is 
$\mathcal{O}(r|\Omega|)$. 
LSQR is mathematically equivalent to conjugate gradient applied to the normal equations. As studied in 
\citep[Section 4]{hayami2018convergence}, the residual error after $k$ iterations decays like 
\[
C \left(\frac{\sigma_{\max}-\sigma_{\min}}{\sigma_{\max}+\sigma_{\min}}\right)^k,  
\]
where $\sigma_{\max}$ and $\sigma_{\min}$ are the largest and smallest non-zero singular values of the rank-deficient matrix
of the least squares problem (\ref{eq:ILS_update_rule_step1}). 
Empirically, at an  oversampling ratio $\rho=2$ and matrices of size $300\times 300$, the above quotient is often smaller than $0.9$.
Thus, LSQR often requires at most a few hundreds of inner iterations to converge. 
For larger sized matrices and more challenging instances, more iterations may be needed for a very accurate solution. We capped
the maximal number of LSQR iterations at 4000. 

\vspace{1em}
\noindent
{\bf Attenuating Non-convergence.} 
On challenging problem instances, where \texttt{R2RILS} fails to find the global minimum solution, empirically the reason is
not convergence to a bad local minima. Instead, due to its global updates, $(U_t,V_t)$ often oscillates in some orbit. To attenuate this behavior, if \texttt{R2RILS} did not coverge within say the first 40 iterations, then once every few iterations we replace the averaging in Eq.~(\ref{eq:ILS_update_rule_step2}) with a weighted averaging, which gives more weight to the previous estimate, of the form
\[
U_{t+1} = \ColNorm\left(\beta U_{t} + \ColNorm\left(\tilde{U}_t\right) \right),
\]
and a similar formula for $V_{t+1}$. In our simulations we took $\beta=1+\sqrt{2}$. Empirically, this modification increased the cases where \texttt{R2RILS} converged within $t_{\max}$ iterations.  


\section{Numerical Results}
        \label{sec:numerical_results}
We present simulation results that demonstrate the performance of \texttt{R2RILS}.
In the following experiments random matrices were generated according to the uniform model. Specifically, $\{u_i\}_{i=1}^{r}, \{v_i\}_{i=1}^{r}$ were constructed by drawing $r$ vectors uniformly at random from the unit spheres in $\mathbb{R}^{m}, \mathbb{R}^{n}$ respectively and then orthonormalizing them. At every simulation we specify the singular values $\{\sigma_{i}\}_{i=1}^r$ and construct the rank $r$ matrix $X_{0}$ as
\[X_0 = \sum_{i =1}^{r}\sigma_i  u_i v_i^{T}.
\]
With high probability, the matrix $X_{0}$ is incoherent \citep{candes2009exact}. 
At each oversampling ratio $\rho$, we generate a random set $\Omega$ of observed entries by flipping a coin with probability $p=\rho \cdot \frac{r(m+n-r)}{(m\cdot n)}$ at each of the $mn$ matrix entries. The size of $\Omega$ is thus variable and distributed as $\mbox{Binom}(m\cdot n,p)$.  
As in \citep{kummerle2018harmonic}, we then verify that each column and row have at least $r$ visible entries and repeat this process until this necessary condition for unique recovery is satisfied. 

We compare \texttt{R2RILS} with maximal number of iterations $t_{\max}=100$ to the following algorithms, 
using the implementations supplied by the respective authors. 
As detailed below, for some of them we slightly tuned their parameters to improve their performance. 
\begin{itemize}
\item \texttt{OptSpace} \citep{keshavan2010matrix}: Maximal number of iterations set to 100. Tolerance parameter $10^{-10}$.
\item \texttt{FPCA} \citep{ma2011fixed}: Forced the implementation to use a configuration for "hard" problem where several parameters are tightened. The two tolerance parameters were set to $10^{-16}$.
\item \texttt{LRGeomCG} \citep{vandereycken2013low}: Executed with its default parameters.
\item \texttt{RTRMC} \citep{boumal2015low}: Maximal number of iterations 300, maximal number of inner iterations set to 500. The gradient tolerance was set to $10^{-10}$.
\item \texttt{ScaledASD} \citep{tanner2016low}: Tolerance parameters $10^{-14}$ and maximal number of iterations 1000. 
\item \texttt{HM-ILS} \citep{kummerle2018harmonic}: Executed with its default parameters.
\end{itemize}

We considered two performance measures. The first is the relative RMSE per-entry over the unobserved entries.
Given an estimated matrix $\hat X$, 
this quantity is defined as
\begin{equation}
	\label{eq:rel-RMSE}
\texttt{rel-RMSE}=\sqrt{\frac{m \cdot n}{|\Omega^{c}|}} \cdot \frac{\|\hat{X} - X_0 \|_{F(\Omega^{c})} }{\| X_0\|_{F}} .
\end{equation}
The second measure is the success probability of an algorithm in the ideal setting of noise-free observations. 
We define success as $\texttt{rel-RMSE} < 10^{-4}$. This is similar to \citep{tanner2016low}, who computed a relative RMSE on all matrix entries, and considered a recovery successful with a slightly looser threshold of $10^{-3}$. 
We compare \texttt{R2RILS} to all the above algorithms except for \texttt{HM-ILS} which will be discussed separately.
In addition, we also tested the \texttt{R3MC} algorithm \citep{mishra2014r3mc} and the damped \texttt{Wiberg} method of \citep{okatani2011efficient}. 
To limit the number of methods shown in the plots of Figs. \ref{fig:test_1} and \ref{fig:test_10}, the performance of these two methods is not shown in these figures. However, their performance was similar to that of other Riemannian optimization based methods.

\vspace{1em}

\noindent
{\bf Well conditioned setting.} In our first experiment, we considered a relatively easy setting with well conditioned matrices of size $1000
\times 1000$ and rank $r=5$, whose
non zero singular values were all set to 1. Figure \ref{fig:test_1} shows the reconstruction ability of various algorithms as a function of the oversampling ratio $\rho$. 
In this scenario all algorithms 
successfully recover the matrix once enough entries are observed. Even in this relatively easy setting, \texttt{R2RILS} shows favorable performance at low oversampling ratios, reaching a relative RMSE around $10^{-14}$.

\vspace{1em}

\noindent
{\bf Mild ill-conditioning.} 
Next, we consider a mild ill-conditioning setting, where the rank $r=5$ matrices 
have a condition number 10, and non-zero singular values $10,8,4,2,1$. As seen in Fig.~\ref{fig:test_10},
 \texttt{R2RILS} is barely affected by this ill-conditioning and continues to recover the underlying matrix with error $10^{-14}$ 
 at oversampling ratios larger than $1.5$. 
 In contrast, \texttt{FPCA}, which performs nuclear norm minimization, recovers the matrix at higher oversampling ratios $\rho>3.4$.
 This is in accordance to similar observations by previous works \citep{tanner2013normalized,kummerle2018harmonic}.
 The other compared algorithms, all of which solve non-convex problems, also require higher oversampling ratios
 than \texttt{R2RILS}, and even then, occasionally fail to achieve a relative RMSE less than $10^{-4}$.

\begin{figure}[t]
\centering
\begin{subfigure}{.5\textwidth}
  \centering
  \includegraphics[width=0.85\textwidth]{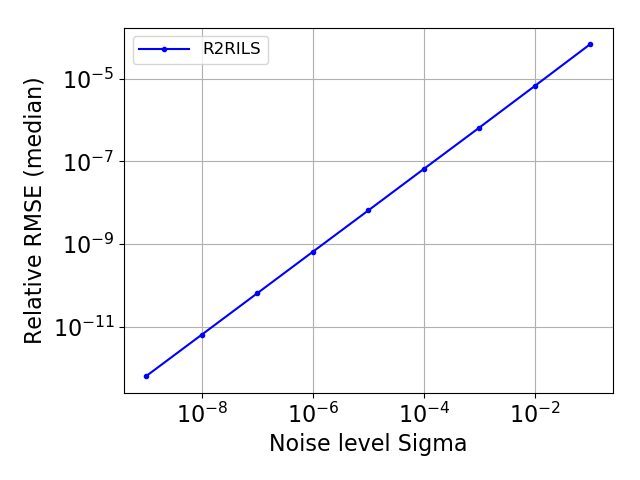}
  \caption{Oversampling 2.5}
  \label{fig:noise_25_RMSE}
\end{subfigure}%
\begin{subfigure}{.5\textwidth}
  \centering
  \includegraphics[width=0.85\linewidth]{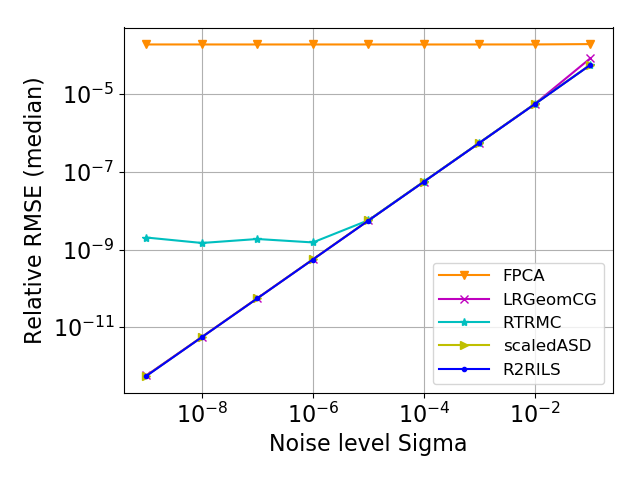}
  \caption{Oversampling 3}
  \label{fig:noise_3_RMSE}
\end{subfigure}
\caption{Comparison of several matrix completion algorithms with ill-conditioned matrices and entries corrupted by additive Gaussian noise.
Matrices were drawn as in the simulations of Fig.~\ref{fig:test_1}. We plot the RMSE per unobserved entry as a function of the standard deviation of the noise.  Each point on the graphs corresponds to  50 independent realizations.	}
\label{fig:test_noise}
\end{figure}

\begin{figure}[t]
\centering
  \includegraphics[width=0.4\textwidth]{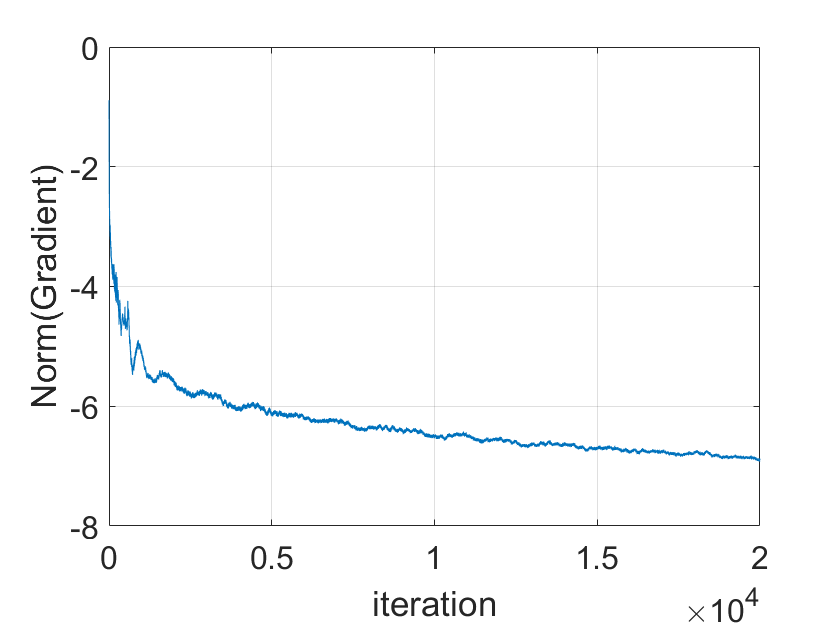}
  \includegraphics[width=0.4\linewidth]{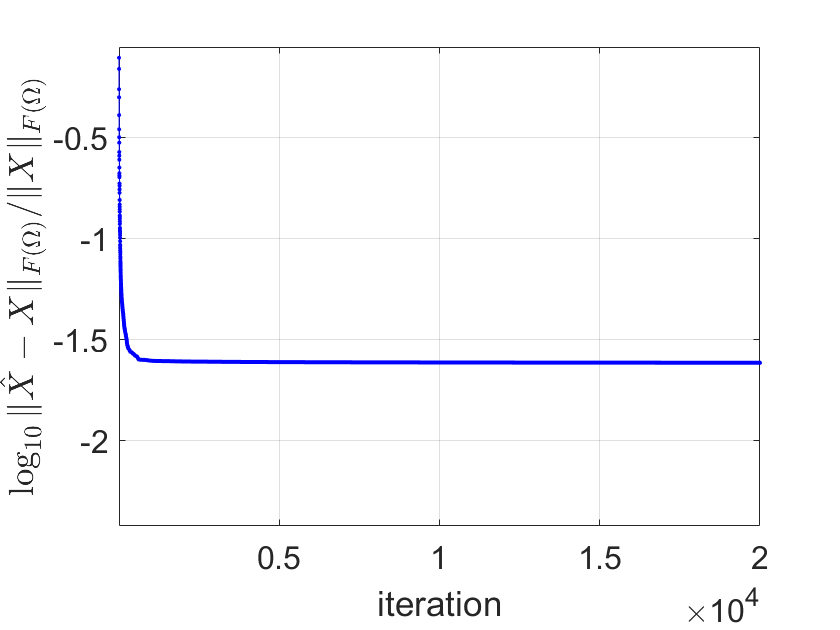}
\caption{Typical optimization path of \texttt{LRGeomCG} on a $1000\times 1000$ matrix of rank $r=5$, condition number 10, at oversampling ratio $\rho=2.2$. (Left) The norm of the gradient at each iteration on a logarithmic scale. (Right) The
normalized error on the observed entries. }
\label{fig:LRGeom}
\end{figure}

\vspace{1em}
\noindent
{\bf Comparison to \texttt{HM-ILS}.} 
The \texttt{HM-ILS} algorithm  \citep{kummerle2018harmonic} was not included in the above simulations due to its 
slow runtime. 
However, from a limited evaluation with smaller sized matrices \texttt{HM-ILS} with Schatten $p$-norm parameter $p=1/2$ has excellent performance, 
comparable to \texttt{R2RILS}, 
also under ill-conditioning. Figure \ref{fig:timing_error_comparison} demonstrates that at a low oversampling ratio $\rho=2.5$, both algorithms perfectly reconstruct matrices of various dimensions. 
Figure \ref{fig:timing_time_comparison} shows that 
\texttt{R2RILS} is faster than \texttt{HM-ILS} by orders of magnitude even for modestly sized matrices.

\vspace{1em}
\noindent
{\bf Comparison to Riemannian Manifold Optimization Methods.} 
We considered the optimization path as well as the 
output of the \texttt{LRGeomCG} algorithm on an instance where it failed to exactly complete the matrix within an increased number of 20000 iterations. Figure~\ref{fig:LRGeom} shows the norm of the gradient and the normalized error on the observed entries of \texttt{LRGeomCG} versus iteration count. 
On this instance, 
the normalized error on the observed entries was 0.0241, whereas on the unobserved entries it was 9.11. 
It seems that at low oversampling ratios, \texttt{LRGeomCG} converges very slowly to a bad local optimum. 
Interestingly, starting from this solution of \texttt{LRGeomCG}, \texttt{R2RILS} recovers  the low rank matrix within
only 20 iterations, with an RMSE of $10^{-13}$. We observed the same behavior for 10 different matrices.
It thus seems that the lifting to a rank $2r$ matrix leads to fewer bad local minima in the optimization landscape. 
A theoretical study of this issue is an interesting topic for further research.

\begin{figure}[t]
	\centering
	\includegraphics[width=.425\textwidth]{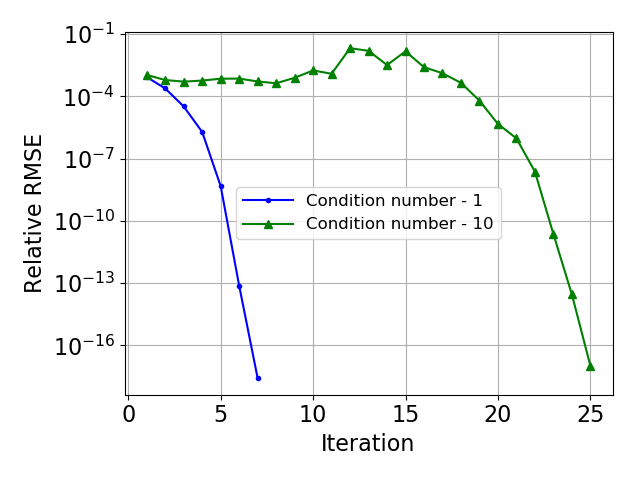}
	\caption{Relative RMSE per {\em observed} entry of \texttt{R2RILS} 
		as a function of the iteration number. Each line represents a single execution of \texttt{R2RILS} on a $1000 \times 1000$ rank-5 matrix, with oversampling ratio $\rho=2.5$. For the condition-1 matrix all non-zero singular values were set to 1. For the matrix with condition number-10 singular values were set to $10, 8, 4, 2, 1$.}
	\label{fig:convergence_rate}
\end{figure}

\vspace{1em}
\noindent
{\bf Stability to Noise.} Figure \ref{fig:test_noise} illustrates the performance of several matrix completion algorithms when i.i.d.~zero mean Gaussian noise is added to every observed entry of $X_0$. As seen in panel \ref{fig:noise_25_RMSE}, \text{R2RILS} is robust to noise even at low oversampling ratios, with
the RMSE linear in the noise level. In Section~\ref{sec:theory}, we prove this result in the case of a fully observed rank-$1$ matrix.
Panel \ref{fig:noise_3_RMSE} shows that most algorithms are also robust to noise, but only at higher oversampling ratios 
where they start to work in the absence of noise.
Algorithms that even without noise failed to recover at an oversampling ratio $\rho=3$ are not included in this graph. 

\vspace{1em}
\noindent
{\bf Convergence Rate.} 
Figure \ref{fig:convergence_rate} illustrates the relative RMSE of \texttt{R2RILS} per observed entry, Eq.~(\ref{eq:rel-RMSE}),  as a function of the iteration number,  on two rank-5 matrices of dimension $1000\times 1000$, oversampling ratio 2.5 and condition numbers 1 and 10. 
It can be observed that \texttt{R2RILS}'s convergence is very quick once it reaches a small enough error. It is also interesting to note that \texttt{R2RILS} does not monotonically decrease the objective in Eq.~(\ref{eq:ILS_update_rule_step1}) at every iteration.

\vspace{1em}
\noindent
{\bf Power-law Matrices.}
Another form of ill conditioning 
occurs when the row and/or column subspaces are only weakly incoherent. 
As in \citep{chen2015completing}, we consider power-law matrices of the form 
$X=D U V^\top D$ where the entries of $U$ and $V$ are i.i.d. $\mathcal N(0,1)$ and $D=D(\alpha)$ is a diagonal matrix
with power-law decay, $D_{ii}=i^{-\alpha}$. When $\alpha=0$ the matrix is highly incoherent and relatively easy to recover. 
When $\alpha=1$ it is highly coherent
and more difficult to complete from few entries. Given a budget on the number of observed entries, a 2-step sampling scheme was proposed in \citep{chen2015completing}. As shown in \citep[Fig. 1]{chen2015completing}, 
$10n\log(n)$ adaptively chosen samples were sufficient to recover 
rank-5 matrices of size $500\times 500$ by nuclear norm minimization, for values of $\alpha\leq 0.9$. In contrast, with entries observed uniformly at random, nuclear norm minimization
required over
$300n\log(n)$ entries for $\alpha\geq 0.8$. 
Figure~\ref{fig:power_law} shows the rate of successful recovery (as in \citep{chen2015completing}, defined as $\|\hat X-X\|_F/\|X\|_F \leq 0.01$) 
by \texttt{R2RILS}
as a function of number of entries, chosen uniformly at random while only verifying that each row and column has at least $r=5$ entries. As seen in the figure, 
\texttt{R2RILS} recovers the underlying matrices with less than $10n\log(n)$ entries up to $\alpha=0.8$,
namely without requiring sophisticated adaptive sampling. 

\begin{figure}[t]
\centering
\includegraphics[width=.4\textwidth]{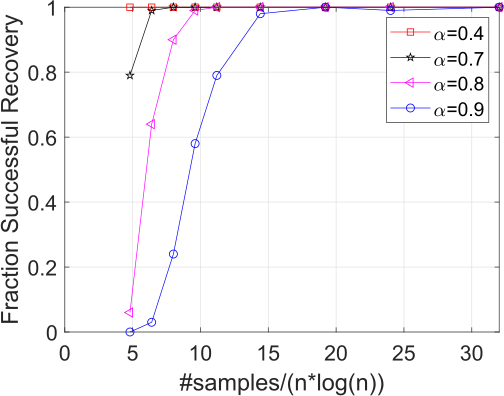}
\caption{Recovery of power-law matrices by \texttt{R2RILS}.}
\label{fig:power_law}
\end{figure}

\vspace{1em}
\noindent
{\bf Datasets from Computer Vision.}
We conclude the experimental section with results on 
two benchmark datasets\footnote{Datasets available at \texttt{https://github.com/jhh37/lrmf-datasets/find/master}} from the computer vision literature
 \citep{buchanan2005damped,hyeong2015secrets}. 
We compare \texttt{R2RILS} with the \texttt{damped Wiberg} 
method\footnote{Code downloaded from http://www.vision.is.tohoku.ac.jp/us/download/.
The method was run with default parameters as suggested in \citep{okatani2011efficient}.},
which is one of the top performing methods on these benchmarks \citep{okatani2011efficient}. 
In these datasets, the goal is to construct a matrix $\hat X$ of given rank $r$, 
whose RMSE at the observed entries is as small as possible. 
The observed indices are structured. Hence, the SVD
of the observed matrix does not provide an accurate initial guess. Instead, it is common to initialize methods 
with random $U$ and $V$, whose entries are i.i.d.~$\mathcal N(0,1)$. The quality of an algorithm is assessed by its ability to
converge to a solution with low observed RMSE, and by the number of iterations and overall clock time it takes to do so, see \citep{hyeong2015secrets}. 

The first dataset is Din (Dino Trimmed in \citep{hyeong2015secrets}), involving the recovery of
the rigid turntable motion of a toy dinosaur. The task is to fit a rank $r=4$ matrix of size $72\times 319$, given 5302 observed entries (23.1\%).
The best known solution has RMSE 1.084673, and condition number 12.4.
Starting from a random guess, \texttt{R2RILS} with maximal number of iterations $t_{\max}=300$ achieved this objective in 99/100 runs. 
The least squares system that \texttt{R2RILS} has to solve is very ill-conditioned for this dataset. We thus also considered a variant, whereby we normalized the columns of the least squares matrix to have unit norm. This variant attained the optimal
objective in all 100 runs, and on average converged in fewer number of iterations. 
Figure \ref{fig:CV} (left) shows 
the cumulative 
rate of attaining this objective for the three tested algorithms. 

The second dataset is UB4 (Um-boy). Here the task is to fit a rank $r=4$ matrix of size $110\times 1760$ given 14.4\% of its entries. 
The best
known solution has observed RMSE 1.266484 and condition number 24.5. Out of 50 runs, \texttt{damped Wiberg} attained this RMSE in 13 cases, and the normalized \texttt{R2RILS} variant in 17 runs.  
The lowest RMSE obtained by \texttt{R2RILS} was slightly higher (1.266597) and this occured in 18 cases. 
This suggests that \texttt{R2RILS} may be improved by preconditioning the linear system. 
A practically important remark is that in these datasets, each iteration of \texttt{damped Wiberg} is significantly faster, 
as it solves a linear system with only $r\min(m,n)$ variables, compared to $r(m+n)$ for \texttt{R2RILS}. 
On the Din dataset,  since the ratio $n/m\approx 4.4$, the mean runtime of the normalized variant is only slightly slower than that of \texttt{damped Wiberg}, 17 seconds versus 11 seconds.
However, on the UB4 dataset, where $n/m=16$, the clock time 
of \texttt{damped Wiberg}
is significantly faster by a factor of about 60. Developing a variant of \texttt{R2RILS} that solves only for the smaller dimension 
and would thus be much faster, is an interesting topic for future research.  

\begin{figure}[t]
\centering
\includegraphics[width=.4\textwidth]{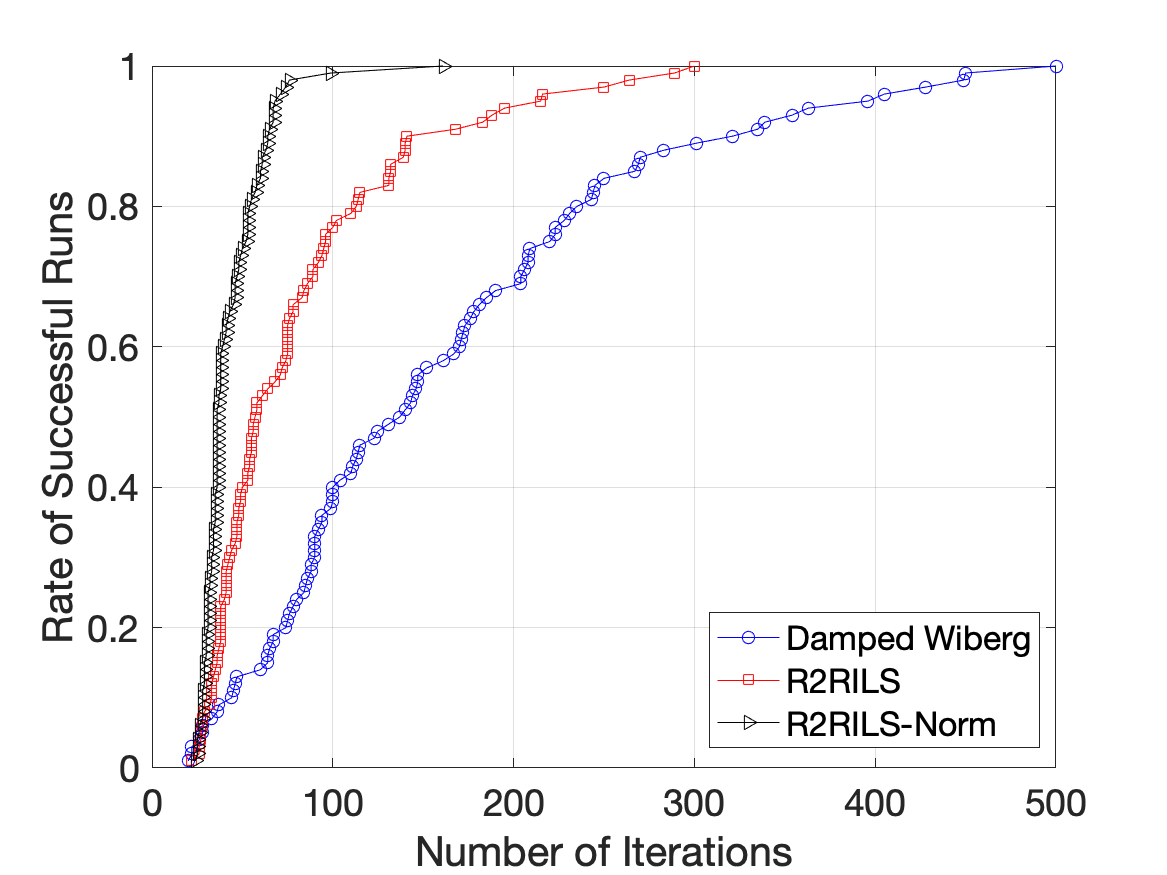}\
\includegraphics[width=.4\textwidth]{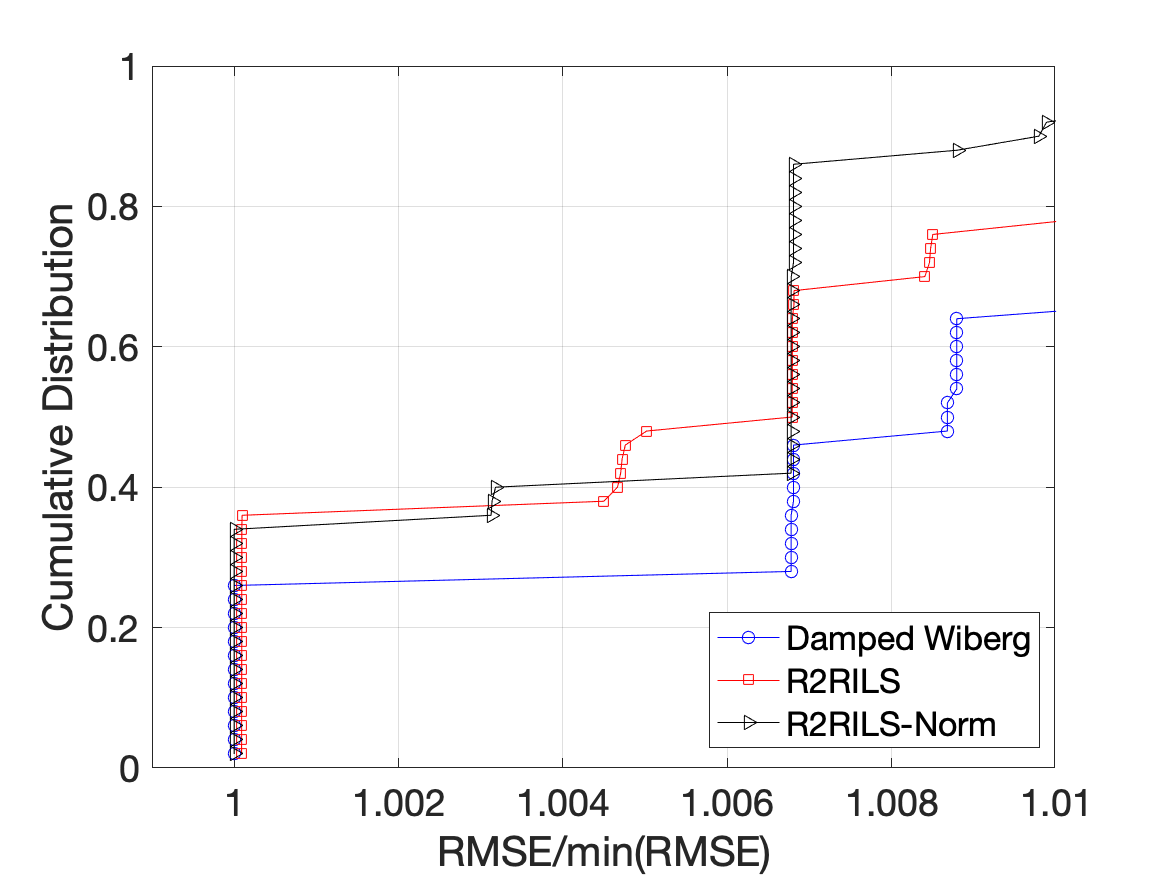}
\caption{(Left) Cumulative distribution for number of iterations to converge to optimal solution on Din dataset. (Right) Cumulative Distribution of 
the observed RMSE per entry, divided by the minimal RMSE,  on the UB4 dataset, from 50 different random initializations. }
\label{fig:CV}
\end{figure}

\section{Theoretical Analysis}
\label{sec:theory}

We present a preliminary theoretical analysis, which provides both motivation and insight into the two steps of \texttt{R2RILS}. First, Lemma \ref{lemma:rank_deficiency} shows that the least squares problem (\ref{eq:ILS_update_rule_step1}) has rank deficiency of dimension at least \(r^{2}\), similar to Wiberg's algorithm \citep[Proposition 1]{okatani2007wiberg}.
Next, even though \texttt{R2RILS} lifts to rank $2r$ matrices, 
Lemma \ref{lemma:critical_point} shows that if it converges, the limiting solution is 
in fact of rank $r$. 
Hence, 
if this solution attains a value of zero for the objective \eqref{eq:ILS_update_rule_step1}, then \texttt{R2RILS} outputs the true underlying matrix $X_0$.
Finally, we study the convergence of \texttt{R2RILS} in the simple rank-1 case. Assuming that the entire matrix is observed, we prove in Theorem \ref{thm:non_asymptotic_convergence} that starting from any initial guess weakly correlated to the true singular vectors, \texttt{R2RILS} converges linearly to the underlying matrix. 
Its proof motivates the averaging step of \texttt{R2RILS}, as it shows that in the rank-1 case, the errors of $(\tilde U_t,\tilde V_t)$ relative to the
true singular vectors are approximately in the {\em opposite} direction compared to the errors in $(U_t,V_t)$. Using this property, we show in Theorem \ref{thm:quadratic_convergence} that locally, the convergence of \texttt{R2RILS} is quadratic.
	Remarks~\ref{rem:optimal_weighting} and \ref{rem:non_minimal_solution} provide insight and theoretical justification 
	for the equal weighting in the averaging step \eqref{eq:ILS_update_rule_step2}, as well as the choice of the minimal norm solution to the least squares problem \eqref{eq:ILS_update_rule_step1}. Finally, Theorem~\ref{thm:noisy} shows that \texttt{R2RILS} is stable to additive noise,  in the fully observed rank-$1$ case.

\begin{lemma}
\label{lemma:rank_deficiency}
Suppose that the \(r\) columns of \(U_{t}\) and of \(V_t\) are both linearly independent. Then, the solution space of Eq.~(\ref{eq:ILS_update_rule_step1}) has dimension at least $r^{2}$.
In addition, in the rank-$1$ case, when $\Omega = [m] \times [n]$ the solution space has dimension exactly $1$.
\end{lemma}
\begin{proof}
The solution of (\ref{eq:ILS_update_rule_step1}) is unique up to the kernel of the linear map
\begin{align} \label{eq:leastSquares_linearMap}
(A,B) \mapsto \Vectorize_{\Omega}\left(U_{t}B^{\top} + AV_{t}^{\top}\right) = \Vectorize_{\Omega}\Big(\sum_{i=1}^{r} (U_{t})_i b_i^{\top} + a_i (V_{t})_i^{\top}\Big),
\end{align}
where $(U_t)_i$ denotes the $i$-th column of $U$ and $\Vectorize_{\Omega}(B)\in\mathbb{R}^{|\Omega|}$ is a vector with entries $B_{i,j}$ for $(i,j) \in \Omega$. 
Choosing $a_i = \sum_{j=1}^{r} \lambda_{i,j} (U_{t})_j$ and $b_i = - \sum_{j=1}^{r} \lambda_{j,i} (V_t)_j$ with  $r^{2}$ free parameters $\lambda_{i,j}$ yields an element of the kernel. Hence,  the dimension of the kernel is at least $r^{2}$.

As for the second part of the lemma, suppose that $(a,b)$ is a non-trivial solution, such that
\begin{equation}
	\label{eq:rank_1_dim_defficiency}
u_{t}b^{\top} + av_{t}^{\top} = 0.
\end{equation}
Then $\exists i$ such that $b_{i} \neq 0$ and by Eq.~(\ref{eq:rank_1_dim_defficiency})
$b_{i}u_{t} = -(v_{t})_{i} a$, 
implying that $a \in \Span\{u_{t}\}$. A similar argument shows that $b \in \Span\{v_{t}\}$. 
Hence the two terms in Eq.~(\ref{eq:rank_1_dim_defficiency}) take the form $u_tb^\top=\lambda_1 u_tv_t^\top$ and
$av_t^\top=\lambda_2 u_tv_t^\top$ for some $\lambda_1,\lambda_2\in\mathbb{R}$. For Eq.~(\ref{eq:rank_1_dim_defficiency}) to hold, $\lambda_1=-\lambda_2$.
Thus, any non trivial solution belongs to the rank-1 subspace $\lambda(u_t,-v_t)$. 
\end{proof}
%

\begin{lemma}
\label{lemma:critical_point}
Let $\mathcal{M}_{r}$ be the manifold of $m\times n$ rank $r$ matrices.
Denote by $L: \mathcal{M}_{r} \rightarrow \mathbb{R}$ the squared error loss on the observed entries, 
\[L(Z) = \| Z - X\|_{F(\Omega)}^{2}.
\]
Suppose that $(U_t, V_t)$ is a fixed point of \texttt{R2RILS}, and that the $r$ columns of $U_t$ and of $V_t$ are both linearly independent. Then the rank $2r$ matrix $\hat X_t$ is 
rank $r$ and
is a critical point of $L$.
\end{lemma}
\begin{proof}
Suppose $(U_{t}, V_{t})$ is a fixed point of \texttt{R2RILS}. Then the solution $(\tilde U_t,\tilde V_t)$ of Eq.~(\ref{eq:ILS_update_rule_step1}) 
can be written as follows, with diagonal matrices $\Sigma_{U}, \Sigma_{V} \in \mathbb{R}^{r \times r}$, 
\[(\tilde{U}_{t}, \tilde{V}_{t}) = (U_{t}\Sigma_{U}, V_{t}\Sigma_{V} )\] 
This implies that \texttt{R2RILS}'s interim rank \(2r\) estimate \(\hat X_t\) of Eq.~(\ref{eq:X_t_2r}) is in fact of rank \(r\), since 
\[\hat{X}_{t} = U_{t}\Sigma_{V}^{\top} V_{t}^{\top} + U_{t} \Sigma_{U} V_{t}^{\top} = U_{t} (\Sigma_{U}  + \Sigma_{V}^{\top}) V_{t}^{\top} .\]
The rank $r$ matrix $\hat X_t$ is a critical point of the function \(L\) on the manifold $\mathcal M_r$ if and only if its gradient $\nabla L$ is orthogonal to the tangent space at $\hat X_t$, denoted \(T_{\hat X_t}\mathcal M_r\). 
Suppose by contradiction that  $\hat X_t$ is not a critical point of $L$ on  $\mathcal{M}_{r}$. Equivalently, 
\begin{equation}
\label{eq:non_ortho_gradient}
\nabla L(\hat{X}_{t}) \centernot{\bot} T_{\hat{X}_t}\mathcal{M}_r .
\end{equation}
The gradient of the loss $L$ at a point $Z$ is given by
$
\nabla L (Z) = 2\left(P_{\Omega}(Z) - X\right)
$
where $P_{\Omega}$ is the projection operator onto the observed entries in the set $\Omega$, 
\[\left(P_{\Omega}(Z)\right)_{i,j} = \begin{cases}
Z_{i,j}, \quad &\mbox{if } (i,j) \in \Omega,\\
0,   &\mbox{otherwise}.
\end{cases}\]
The tangent space at a point $Z = U \Sigma V^{\top} \in \mathcal{M}_{r}$, where $\Sigma$ is an invertible $r\times r$ matrix,  is given by \citep[Proposition 2.1]{vandereycken2013low}
\begin{align*}
T_{Z}\mathcal{M}_r = \left\{ 
UB^{\top} + AV^{\top}  \; | \;  A \in \mathbb{R}^{m \times r}, B \in \mathbb{R}^{n \times r}
\right\}.
\end{align*}
Eq.~(\ref{eq:non_ortho_gradient}) means that the projection  
 of $\nabla L(\hat{X}_t)$ onto the tangent space $T_{\hat{X}_t}\mathcal{M}$ is non trivial. 
Let $U_t B^\top + AV_t^\top$ be this projection.  Then
\[
\| P_{\Omega}(\hat{X}_{t}) - X  - U_{t} B^{\top} - AV_{t}^{\top}\|_{F} < \| P_{\Omega}(\hat{X}_{t}) -  X \|_{F}.
\]
Since $P_{\Omega}(\hat{X}_t)$ and $X$ both vanish on $\Omega^c$, the right hand side equals
$\| P_{\Omega}(\hat{X}_{t}) -  X \|_{F(\Omega)}$.
Thus, 
\begin{align*}
\| P_{\Omega}(\hat{X}_{t}) - X - U_t B^{\top} - AV_t^{\top}\|_{F(\Omega)}  &\leq \| P_{\Omega}(\hat{X}_{t}) - X - U_t B^{\top} - A V_t^{\top}\|_{F}
< \| P_{\Omega}(\hat{X}_{t}) - X \|_{F(\Omega)}.
\end{align*}
This contradicts the assumption that $(\tilde{U}_{t}, \tilde{V}_{t})$ is a global minimizer of Eq.~(\ref{eq:ILS_update_rule_step1}). 
\end{proof}

\begin{corollary}\label{corollary:convergence_to_X0}
Consider a noise-free matrix completion problem with an underlying matrix $X_0$ of rank $r$ and a set $\Omega \subset [m]\times [n]$ such that the solution to \eqref{eq:LRMC} is unique. If \texttt{R2RILS} converged to a fixed point $(U_t, V_t)$ with a zero value for the least squares objective \eqref{eq:ILS_update_rule_step1}, $\left\|U_t \tilde V_t^\top + \tilde U_t V_t^\top - X \right\|_{F(\Omega)} = 0$, then its output is the true underlying matrix $X_0$.
\end{corollary}
\begin{proof}
By Lemma~\ref{lemma:critical_point}, if \texttt{R2RILS} converged then the intermediate matrix $U_t \tilde V_t^\top + \tilde U_t V_t^\top$ is in fact rank-$r$. As this matrix attains a zero value for the objective \eqref{eq:LRMC}, it equals $X_0$.
\end{proof}


Next, we study the convergence of \texttt{R2RILS}, in the simple case where 
\(X_{0}=\sigma u v^\top\) is of rank-1, and assume that we have observed {\em all} entries of $X_0$.
These two assumptions allow us to find a closed form solution to the least squares problem and are critical for our proof analysis. It may be possible to extend our proof to higher rank settings and to partially observed matrices with a more complicated proof. We leave this for future work. 

Let $(u_t,v_t)$ be the estimates of \texttt{R2R2ILS} at iteration $t$. 
In the rank-$1$ case $u_t$ and $v_t$ are vectors and we may decompose each of them into two components. The first is their projection on the true $(u, v)$, and the second is the orthogonal complement which is their error, 
\begin{equation}
	\label{eq:u1_v1}
u_t = \alpha_t u + \sqrt{1-\alpha_t^2} e_{u,t}  \; , \; v_t =\beta_t v + \sqrt{1-\beta_t^2} e_{v,t}.
\end{equation}
For future use we define  $\epsilon_t = \sqrt{1-\alpha_t^2}$, $\delta_t = \sqrt{1-\beta_t^2}$, and the following two quantities
\begin{equation}
h(\epsilon) = \sqrt{1+2\epsilon^2-3\epsilon^4},\quad \quad r(\epsilon) = \frac{1+\epsilon^2+h(\epsilon)}{\sqrt{2(1+3\epsilon^2+h(\epsilon))}} .
	\label{eq:h_r}
\end{equation}
Note that both $(u,v)$ and hence also $(\alpha_t,\beta_t)$ are determined up a joint $\pm 1$ sign. 
In the following, we will assume that the initial guess satisfies 
$\text{sign}(\alpha_1) = \text{sign}(\beta_1)>0$.

The first theorem below shows that in the fully-observed rank-1 case, from any initial guess weakly correlated to the true singular vectors ($\alpha_1,\beta_1>0$),  \texttt{R2RILS} converges linearly to the matrix $X_0$. The second theorem shows that asymptotically, the convergence is quadratic.

\begin{theorem}
\label{thm:non_asymptotic_convergence}
Assume that $\Omega = [m]\times[n]$, and that the initial guess $(u_1,v_1)$ satisfies $\alpha_1,\beta_1>0$.
Then the sequence of estimates $(u_{t}, v_{t})$ generated by \texttt{R2RILS} converges to $(u, v)$ linearly with a contraction factor smaller than $\sqrt{1-\frac{1}{\sqrt{2}}} \approx 0.54$. Specifically, 
\begin{align}\label{eq:non_asymptotic_convergence}
\left\| \begin{pmatrix} u_{t+1} \\  v_{t+1} \end{pmatrix} - \begin{pmatrix} u \\  v \end{pmatrix} \right\| = R(\epsilon_t, \delta_t) \left\| \begin{pmatrix} u_t \\  v_t \end{pmatrix} - \begin{pmatrix} u \\  v \end{pmatrix} \right\|
\end{align}
where 
$$
R(\epsilon_t, \delta_t) = \left[\frac{2 - r(\epsilon_t) - r(\delta_t)}{2 - \sqrt{1-\epsilon_t^2} - \sqrt{1-\delta_t^2}}\right]^\frac{1}{2} 
\leq \sqrt{1-\frac{1}{\sqrt{2}}} \max\{\epsilon_t, \delta_t\}, 
$$
and the function $r(\epsilon)$ was defined in Eq.~(\ref{eq:h_r}) above. 
\end{theorem}

\begin{theorem}
\label{thm:quadratic_convergence}
As $\epsilon_t, \delta_t \to 0$, the convergence is quadratic, with an asymptotic contraction factor of $\sqrt{\epsilon_t^4 - \epsilon_t^2 \delta_t^2 + \delta_t^4} \leq \sqrt{2} \max\{\epsilon_t^2, \delta_t^2\}$.
\end{theorem}

\begin{remark} In the fully observed rank-1 case, it follows from the proof of Theorem \ref{thm:non_asymptotic_convergence} and can also be observed empirically that if the  initial guess $(u_1,v_1)$ is misaligned with the singular vectors $(u,v)$, namely $\alpha_1\cdot \beta_1<0$, then \texttt{R2RILS} fails to converge. However, this has little practical significance, since if only some entries are observed and/or the rank is higher then with enough observed entries, empirically \texttt{R2RILS} converges to the exact low rank matrix. 
\end{remark}

Next, we present two remarks that theoretically motivate the choice of the minimal norm solution and the averaging step of \texttt{R2RILS}, both in the fully observed rank-$1$ case.

\begin{remark}
\label{rem:optimal_weighting}
Consider a variant of \texttt{R2RILS}, where step II is replaced by
\begin{align}\begin{aligned}\label{eq:R2RILS_variant}
u_{t+1} = \ColNorm\left(u_t + w_u\cdot \ColNorm\left(\tilde{u}_t\right)\right), \,
v_{t+1} = \ColNorm\left(v_t + w_v\cdot \ColNorm\left(\tilde{v}_t\right)\right),
\end{aligned}\end{align}
with weights $w_u, w_v \in \mathbb{R}$ that possibly depend on $t$. In the case $\Omega = [m] \times [n]$ and $\alpha_1, \beta_1 > 0$, it is possible to converge to the exact rank-$1$ matrix $X_0$ after a single iteration, by setting
\begin{align} \begin{aligned} \label{eq:optimal_weights}
w_u = \left(u_t^\top \ColNorm\left(\tilde{u}_t\right)\right)^{-1},\quad
w_v = \left(v_t^\top \ColNorm\left(\tilde{v}_t\right)\right)^{-1}.
\end{aligned} 
\end{align}
Note that as $\epsilon_t, \delta_t \to 0$, namely as the current estimate becomes closer to $(u,v)$, then the optimal weights satisfy $w_u, w_v\to 1$, resulting in the original \texttt{R2RILS}. In other words, the averaging step of \texttt{R2RILS} is asymptotically optimal in the fully observed rank-1 case.
\end{remark}

\begin{remark}\label{rem:non_minimal_solution}
Assume $\Omega = [m]\times [n]$ and let $X = \sigma uv^\top$. By the proof of Lemma~\ref{lemma:critical_point}, in this case,  
the kernel of the linear map \eqref{eq:leastSquares_linearMap} is rank one, spanned by $(u_t,-v_t)$. 
Consider a variant where in step I,
instead of the minimal norm solution $(\tilde{u}_t,\tilde{v}_t)$, it takes a solution of the form
\[\label{eq:general_non_minimal_solution}
\begin{pmatrix} \hat{u}_t \\ \hat{v}_t \end{pmatrix} = \begin{pmatrix} \tilde{u}_t \\ \tilde{v}_t \end{pmatrix} + \lambda_t \begin{pmatrix} u_t \\ -v_t \end{pmatrix}
\]
where the scalar $\lambda_t$ may in general depend on $u_t,v_t$ and on $t$. 
Assume that  $(u_t,v_t)$ are aligned with $(u,v)$, namely $\alpha_t,\beta_t>0$. 
Then, the {\em only} choice that leads to quadratic convergence of the column or row space as $\epsilon_t\to 0$ or $\delta_t \to 0$ is the minimal norm solution, namely $\lambda_t=0$.
\end{remark}

Finally, the following theorem shows that in the presence of additive Gaussian noise, the estimates $(u_t, v_t)$ of \texttt{R2RILS} are close to the true vectors $(u,v)$ up to an error term linear in the noise level. 
A similar result, with different constants, also holds for  sub-Gaussian noise.

\begin{theorem}\label{thm:noisy}
Let $\Omega = [m]\times [n]$, and assume w.l.o.g. that $\rho \equiv \sqrt{\frac mn} \leq 1$. Let $X = \sigma u v^\top + Z$ where $\sigma>0$, $\|u\|=\|v\|=1$ and all entries $Z_{i,j}$ are i.i.d.~Gaussian with mean 0 and variance $\eta_0^2$.
Let $\delta \in \left(0, \frac 14 \right]$, and assume the initial guess of \texttt{R2RILS}, $(u_1, v_1)$, satisfies $\alpha_1, \beta_1 \geq \delta$.
Denote the normalized noise level $\eta \equiv \eta_0 \sqrt n$, and the constants $R \equiv \sqrt{1-\frac{1}{\sqrt 2}} \simeq 0.54$, $C \equiv \frac{50}{1-R}$. If 
\begin{align}\label{eq:lowNoiseLevel_condition}
\frac{\eta}{\sigma} \leq \frac{\sqrt 2 }{C} \cdot \delta,
\end{align}
then with probability at least $1 - e^{-\frac{n}{2}}$, for all $t \geq 2$,
\begin{align}
\left\| \begin{pmatrix} u_t \\ v_t \end{pmatrix} - \begin{pmatrix} u \\ v \end{pmatrix} \right\| &\leq \sqrt 3 R^{t-2} + 4C \left(1 - R^{t-2}\right) \frac{\eta}{\sigma} . \label{eq:noisyEstimatesError}
\end{align}
Hence, after $\log(\eta/\sigma)$ iterations, the error is $O(\eta/\sigma)$.
\end{theorem}

We first present the proof of Theorem \ref{thm:quadratic_convergence} assuming Theorem \ref{thm:non_asymptotic_convergence} holds, 
and then present the proof of the latter. The proof of Theorem~\ref{thm:noisy} appears in Appendix~\ref{sec:appendix_proofOfNoisyTheorem}.

\begin{proof}[Proof of Theorem \ref{thm:quadratic_convergence}]
Since $h(\epsilon) = 1 + \epsilon^2 - 2\epsilon^4 + 2\epsilon^6 + \mathcal{O}\left(\epsilon^8\right)$, it follows that $r(\epsilon) = 1 - \frac{1}{2} \epsilon^6 + \mathcal{O}\left(\epsilon^8\right)$. Substituting these Taylor expansions into $R(\epsilon_t, \delta_t)$ gives
\[
R(\epsilon_t, \delta_t) = \left[\frac{\epsilon_t^6 + \delta_t^6 + \mathcal{O}\left(\epsilon_t^8\right)+\mathcal{O}\left(\delta_t^8\right)}{\epsilon_t^2 + \delta_t^2 + \mathcal{O}\left(\epsilon_t^4\right)+\mathcal{O}\left(\delta_t^4\right)}\right]^\frac{1}{2} = \sqrt{\epsilon_t^4-\epsilon_t^2\delta_t^2+\delta_t^4}\,\cdot (1+o(1)) .
\]
\end{proof}

To prove Theorem~\ref{thm:non_asymptotic_convergence}, we  use the following lemma which provides a closed form 
expression to the minimal norm solution of Eq.~\eqref{eq:ILS_update_rule_step1}, in the rank-1 case where all entries of the matrix have been observed. 
The proof of this auxiliary lemma appears in Appendix \ref{sec:appendix_proofOfMinNormSolLemma}. 

\begin{lemma}\label{lem:minNormSol_generalX} 
	Assume that $\Omega=[m]\times [n]$. Given non-zero vectors $(u_t,v_t)\in \mathbb{R}^m\times \mathbb{R}^n$
and an observed matrix $X$, the minimal norm solution to the least squares problem \eqref{eq:ILS_update_rule_step1} is
\begin{align}
\tilde{u} = \frac{1}{\|{v}_t\|^2} \left(X{v}_t - \frac{{u}_t^\top X {v}_t}{N_t} {u}_t\right), \quad
\tilde{v} = \frac{1}{\|u_t\|^2} \left(X^\top {u}_t - \frac{{u}_t^\top X {v}_t}{N_t} {v}_t\right). 
		\label{eq:minNormSol_generalX}
\end{align}
where $N_t=\|u_t\|^2 + \|v_t\|^2$. If the matrix $X$ is rank-one, $X = \sigma uv^\top$ with $\sigma>0$, 
and all vectors are normalized, $\|u\|=\|v\|=\|u_t\|=\|{v}_t\|=1$, Eq.~(\ref{eq:minNormSol_generalX}) simplifies to 
\begin{align}\label{eq:minNormSol_rank1X}
\begin{aligned}
\tilde{u} = \left(u - \tfrac 12 \alpha_t u_t\right)\beta_t \sigma, \quad
\tilde{v} = \left(v - \tfrac 12 \beta_t v_t\right)\alpha_t \sigma
\end{aligned}
\end{align}
where $\alpha_t = {u}^\top {u}_t$ and $\beta_t = v^\top {v}_t$.
\end{lemma}


\begin{proof}[Proof of Theorem~\ref{thm:non_asymptotic_convergence}]
It follows from Lemma~\ref{lem:minNormSol_generalX}, that the result of step I 
of \texttt{R2RILS} 
is given by Eq.~\eqref{eq:minNormSol_rank1X}. 
The normalized result is thus
\begin{align}\label{eq:tilde_normalized_withSign}
\ColNorm\left(\tilde{u}_t\right) = \frac{u - \frac 12 \alpha_t u_t}{\tilde{\alpha}_t} \text{sign}(\beta_t), \quad
\ColNorm\left(\tilde{v}_t\right) = \frac{v - \frac 12 \beta_t v_t}{\tilde{\beta}_t} \text{sign}(\alpha_t),
\end{align}
where $\tilde \alpha_t = \sqrt{1-\frac{3\alpha_t^2}{4}}$ and $\tilde \beta_t = \sqrt{1-\frac{3\beta_t^2}{4}}$.
Assume for now that $\alpha_t, \beta_t > 0$. Later we will prove that $\alpha_1, \beta_1 > 0$ indeed implies $\alpha_t, \beta_t > 0$ for all $t\geq 1$. Then 
the next estimate is 
\begin{align*}
u_{t+1} &= \ColNorm\left(u_t + \ColNorm(\tilde{u}_t) \right) = \frac{u + (\tilde{\alpha}_t-\frac 12 \alpha_t) u_t}{\sqrt{2-\frac 32 \alpha_t^2+\alpha_t\tilde{\alpha}_t}}
\end{align*}
with a similar expression for $v_{t+1}$.
Their projections on the true vectors are
\begin{align}\begin{aligned}\label{eq:alphabeta_nextEstimate}
\alpha_{t+1} = u^\top u_{t+1} = \frac{1+\left(\tilde{\alpha_t}-\frac 12 \alpha_t\right)\alpha_t}{\sqrt{2-\frac 32 \alpha_t^2+\alpha_t\tilde{\alpha}_t}}, \quad
\beta_{t+1} = v^\top v_{t+1} = \frac{1+\left(\tilde{\beta}_t- \frac 12 \beta_t\right)\beta_t}{\sqrt{2 - \frac 32 \beta_t^2 + \beta_t \tilde\beta_t}}. 
\end{aligned}\end{align}
The norm of the next estimate error is thus
\begin{align*}
E_{t+1} &= \left\| \begin{pmatrix} u_{t+1} \\  v_{t+1} \end{pmatrix} - \begin{pmatrix} u \\  v \end{pmatrix} \right\| 
= \sqrt{2\left(2 - \alpha_{t+1} - \beta_{t+1}\right)}.
\end{align*}
Similarly, $E_t =\sqrt{2\left(2 - \alpha_t - \beta_t\right)}$. Finally we plug $\epsilon_t = \sqrt{1-\alpha_t^2}$ and $\delta_t = \sqrt{1-\beta_t^2}$. With $h(\epsilon_t)$ and $r(\epsilon_t)$ as defined in \eqref{eq:h_r},
$\left(\tilde\alpha_t - \frac{\alpha_t}{2}\right)\alpha_t = \epsilon_t^2 + h(\epsilon_t)$ and $-\frac{3}{2}\alpha_t^2 + \alpha_t \tilde\alpha_t = 2\left(3\epsilon_t^2+h(\epsilon_t)\right)$,
implying $\alpha_{t+1}=r(\epsilon_t)$ and $\beta_{t+1} = r\left(\delta_t\right)$. Thus $\frac{E_{t+1}}{E_t} = R(\epsilon_t, \delta_t)$.

To conclude the proof, we prove by induction that $\alpha_t, \beta_t > 0$ for all $t\geq 1$. At $t=1$, the condition is one the assumptions of the theorem. Next, assuming $\alpha_t, \beta_t > 0$ yields Eq.~\eqref{eq:alphabeta_nextEstimate}. Since $\tilde{\alpha}_t-\alpha_t/2\geq 0$ for any $\alpha_t\in[0,1]$ then $\alpha_{t+1} > 0$. A similar proof holds for $\beta_{t+1} $. 
\end{proof}

\begin{proof}[Proof of Remark~\ref{rem:optimal_weighting}]
Note that, following Eqs.~\eqref{eq:tilde_normalized_withSign}, $w_u = \frac{2\tilde{\alpha}_t}{\alpha_t}$ and $w_v = \frac{2\tilde{\beta}_t}{\beta_t}$.
Hence, the next estimate of the modified algorithm \eqref{eq:R2RILS_variant} is
\begin{align*}
u_{t+1} &= \ColNorm\left(u_t + w_u \ColNorm\left(\left(u - 
\tfrac{\alpha_t u_t}2
\right)\beta_t \sigma\right)\right) 
= \ColNorm\left(u_t +
  \tfrac{\tilde{\alpha}_t}{\alpha_t} \tfrac{2u - \alpha_t u_t}{\tilde{\alpha}_t}\right) 
 = u
\end{align*}
and similarly $v_{t+1} = v$. Since $\alpha_t = \sqrt{1-\epsilon_t^2} \to 1$ as $\epsilon_t\to 0$ and $w_u = \frac{2}{\alpha_t}\sqrt{1-\frac{3\alpha_t^2}{4}} \to 1$ as $\alpha_t\to 1$, and similarly $w_v \to 1$ as $\delta_t\to 0$, the second part of the remark follows.
\end{proof}

\begin{proof}[Proof of Remark~\ref{rem:non_minimal_solution}]
We focus on the column space part. The proof for the row space part is similar.
The minimal norm solution $\tilde u_t$ is given by \eqref{eq:minNormSol_rank1X}. Hence, $\ColNorm(\tilde u_t)$ 
cancels the factor $\beta_t \sigma$, decoupling it from the row estimate $v_t$ and the singular value of $X$. 
 In contrast, the column space part of the next estimate of the non-minimal norm solution \eqref{eq:general_non_minimal_solution} is 
\begin{align*}
\hat{u}_{t+1} &= \ColNorm\left( u_t + \ColNorm\left(\tilde u_t + \lambda_t u_t  \right)\right) \nonumber 
= \frac{\beta'_t u + \left(\gamma - \frac 12 \alpha_t \beta'_t +  \lambda_t  \right) u_t}{\sqrt{\tilde{\alpha}_t^2 {\beta'_t}^2 + (\alpha_t \beta'_t + \gamma + \lambda_t) (\gamma + \lambda_t)}} 
\end{align*}
where $\beta'_t = \beta_t \sigma$, $\gamma = \sqrt{\tilde{\alpha}_t^2{\beta'_t}^2 + \lambda_t \alpha_t \beta'_t + \lambda_t^2}$ and $\tilde{\alpha}_t = \sqrt{1-\frac{3\alpha_t^2}{4}}$. Therefore, 
\begin{align*}
\hat{\alpha}_{t+1} &= u^\top \hat{u}_{t+1} = \frac{\beta'_t + \left( \gamma - \frac 12 \alpha_t \beta'_t + \lambda_t \right) \alpha_t}{\sqrt{\tilde{\alpha}_t^2 {\beta'_t}^2 + (\alpha_t \beta'_t + \gamma + \lambda_t) (\gamma + \lambda_t)}} \\
&= 1 - \frac{4\lambda_t^2}{(\beta'_t + 2\lambda)^2} (1-\alpha_t) - \frac{2\beta'_t \lambda_t\left(4{\beta'_t}^2 - \beta'_t\lambda_t + 4\lambda_t^2\right)}{(\beta'_t + 2\lambda)^4}\left(1 - \alpha_t\right)^2  + \mathcal{O}\left((1 - \alpha_t)^3\right) .
\end{align*}
It is easy to verify that $\hat \alpha_{t+1} = 1 + \mathcal{O}\left((1-\alpha_t)^3\right)$ if and only if $\lambda_t = 0$. In fact, for any $\lambda_t \neq 0$, $\hat \alpha_{t+1} = 1 + \Theta\left(1-\alpha_t\right)$. Indeed, $\hat \alpha_{t+1} = 1 + \mathcal{O}\left((1-\alpha_t)^3\right)$ is a necessary and sufficient condition for quadratic convergence of the column space part,  $\frac{\left\|\hat u_{t+1} - u \right\|}{\left\|u_t - u \right\|} = \mathcal{O}\left(\epsilon_t^2\right)$, since $\left\| \hat u_{t+1} - u \right\| = \sqrt{2(1 - \hat \alpha_{t+1})}$, $\left\| u_t - u \right\| = \sqrt{2(1 - \alpha_t)}$ and
$\epsilon_t^2 = 1-\alpha_t^2 = 2(1-\alpha_t) + \mathcal{O}\left((1-\alpha_t)^2\right)$.
\end{proof}

\appendix

\section{Proof of Theorem~\ref{thm:noisy}}\label{sec:appendix_proofOfNoisyTheorem}

First, we present two auxiliary lemmas regarding \texttt{R2RILS} in the presence of noise. 

\begin{lemma}\label{lem:noisyAlphaBeta_geq_delta}
Under the conditions of Theorem~\ref{thm:noisy}, with probability at least $1 - e^{-\frac{n}{2}}$,
\begin{align}\label{eq:noisyAlphaBeta_geq_delta}
{ \alpha_t, \beta_t \geq \frac 14, \quad \forall t\geq 2. }
\end{align}
\end{lemma}

\begin{lemma}\label{lem:nextNoisyEstimateDeviation_assumingAlphaBetaGetDelta}
Denote by $\left(u_{t+1}, v_{t+1}\right)$ and $\left(u_{t+1}^{(0)}, v_{t+1}^{(0)}\right)$ the next estimate of \texttt{R2RILS} in the presence and in the absence of noise, starting from $(u_t,v_t)$.
Let $\delta \in \left(0, \frac 14\right]$ and assume $(u_t,v_t)$ satisfies $\alpha_t, \beta_t \geq \delta$. 
Then under the conditions of Theorem~\ref{thm:noisy}, w.p.~at least $1 - e^{-\frac{n}{2}}$,
\begin{align}
\left\| u_{t+1} - u_{t+1}^{(0)} \right\| \leq \frac{50}{\sqrt 2 \delta} \frac{\eta}{\sigma}
\quad \mbox{and} \quad
\left\| v_{t+1} - v_{t+1}^{(0)} \right\| \leq \frac{50}{\sqrt 2 \delta} \frac{\eta}{\sigma}. \label{eq:nextNoisyEstimateDeviation_uvApart}
\end{align}
\end{lemma}

\begin{proof}[Proof of Theorem~\ref{thm:noisy}]
Denote by $E_t$ the $\ell_2$ estimation error of \texttt{R2RILS} at iteration $t$,
\begin{align*}
E_t = \left\| \begin{pmatrix} u_t \\ v_t \end{pmatrix} - \begin{pmatrix} u \\ v \end{pmatrix} \right\|
  = \sqrt{2(1-\alpha_t) + 2(1-\beta_t)}
\end{align*}
{Since by Lemma~\ref{lem:noisyAlphaBeta_geq_delta} $\alpha_2, \beta_2 \geq \frac 14$, then 
	$E_2  \leq \sqrt 3$}. Let $u_{t+1}, v_{t+1}, u_{t+1}^{(0)}$ and $v_{t+1}^{(0)}$ be defined as in Lemma~\ref{lem:nextNoisyEstimateDeviation_assumingAlphaBetaGetDelta}.
Combining Lemma~\ref{lem:noisyAlphaBeta_geq_delta} and Theorem~\ref{thm:non_asymptotic_convergence} gives that for the noiseless update
\begin{align}\label{eq:contraction}
\left\| \begin{pmatrix} u_{t+1}^{(0)} \\  v_{t+1}^{(0)} \end{pmatrix} - \begin{pmatrix} u \\  v \end{pmatrix} \right\| \leq R \left\| \begin{pmatrix} u_t \\  v_t \end{pmatrix} - \begin{pmatrix} u \\  v \end{pmatrix} \right\|
= R E_t.
\end{align}
Combining Lemma~\ref{lem:noisyAlphaBeta_geq_delta} and Lemma~\ref{lem:nextNoisyEstimateDeviation_assumingAlphaBetaGetDelta} {with $\delta = \frac 14$} gives that with probability at least $1-e^{-\frac n2}$
\begin{align}\label{eq:deviation}
\left\| \begin{pmatrix} u_{t+1} \\ v_{t+1} \end{pmatrix} - \begin{pmatrix} u_{t+1}^{(0)} \\ v_{t+1}^{(0)} \end{pmatrix} \right\| 
	\leq 200 \frac{\eta}{\sigma} = 4 C (1-R)\frac{\eta}{\sigma}.
\end{align}
Combining Eqs.~\eqref{eq:contraction}, \eqref{eq:deviation} and the triangle inequality gives
that $E_{t+1} \leq R E_t + 4C(1-R)\frac{\eta}{\sigma}$.
Iteratively applying this recurrence relation, and the bound $E_2\leq \sqrt 3$, yields Eq.~\eqref{eq:noisyEstimatesError}.
\end{proof}

\begin{proof}[Proof of Lemma~\ref{lem:noisyAlphaBeta_geq_delta}]
In the absence of noise, by Theorem~\ref{thm:non_asymptotic_convergence}, $\alpha_t, \beta_t \geq \frac 14$ for all $t\geq 2$ as each iteration of \texttt{R2RILS} brings the estimate $(u_t,v_t)$ closer to the singular vectors $(u,v)$. 
Eq.~\eqref{eq:noisyAlphaBeta_geq_delta} holds also in the presence of low noise 
since the improvement in the estimate is larger than the effect of the noise.
We prove this by induction. Since by assumption $\alpha_1,\beta_1\geq \delta$ and $\delta \leq \frac 14$, it suffices to show that if $\alpha_t,\beta_t\geq \delta$
then $\alpha_{t+1},\beta_{t+1}\geq 1/4$. 
{
From Eq.~\eqref{eq:alphabeta_nextEstimate}, for any $t\geq 1$, the minimal value of $\alpha_{t+1}^{(0)} = {u_{t+1}^{(0)}}^\top u$ is $\frac{1}{\sqrt{2}}$, and thus
$\left\| u_{t+1}^{(0)} - u \right\| = \sqrt{2\left(1-\alpha_{t+1}^{(0)}\right)} \leq R\sqrt{2}$. 
Combining Lemma~\ref{lem:nextNoisyEstimateDeviation_assumingAlphaBetaGetDelta} and assumption \eqref{eq:lowNoiseLevel_condition} gives that with probability at least $1-e^{-\frac n2}$,
$\left\| u_{t+1} - u_{t+1}^{(0)} \right\| \leq \frac{50}{\sqrt{2} \delta} \frac{\eta}{\sigma} \leq \frac{50}{C} = 1-R $.
Hence, by the triangle inequality,
\begin{align*}
\sqrt{2(1-\alpha_{t+1})} &= \left\| u_{t+1} - u \right\| \leq \left\|u_{t+1} - u_{t+1}^{(0)}\right\| + \left\|u_{t+1}^{(0)} - u \right\| \leq 1-R + R\sqrt 2,
\end{align*}
which implies $\alpha_{t+1} > \frac{1}{4}$. The proof for $\beta_{t+1}$  is similar.}
\end{proof}

Finally, to prove Lemma~\ref{lem:nextNoisyEstimateDeviation_assumingAlphaBetaGetDelta}, we use the following lemma
on the largest singular value $\sigma_1(\bar{Z})$ of a Gaussian random matrix $\bar{Z}$, see \citep[Theorem 2.13]{davidson_szarek_2001}.

\begin{lemma}
\label{lem:Zbar_maximalSV}
Let $m, n \in \mathbb{N}$ be such that $m \leq n$, and denote $\rho \equiv \sqrt{\frac mn} \leq 1$. Let $\bar Z \in \mathbb{R}^{m\times n}$ be a matrix whose entries $\bar Z_{i,j}$ are all i.i.d. $\mathcal N(0,1/n)$.
Then for any $s \geq 0$
\begin{align}\label{eq:Zbar_maximalSV}
\mathbb{P}\left[\sigma_1\left(\bar Z\right) \leq 1 + \rho + s\right] \geq 1 - e^{-\frac{n s^2}{2}}.
\end{align}
\end{lemma}

\begin{proof}[Proof of Lemma~\ref{lem:nextNoisyEstimateDeviation_assumingAlphaBetaGetDelta}]
We prove the bound on $u_{t+1}$. 
%
The proof for $v_{t+1}$ is similar.
Given 
$(u_t, v_t)$, 
at iteration $t+1$, 
\texttt{R2RILS} calculates three quantities: 
the minimal norm solution $\tilde u_t$ to \eqref{eq:ILS_update_rule_step1}; its normalized version $\bar u_t \equiv \frac{\tilde u_t}{\left\|\tilde u_t\right\|}$; and the next estimate, 
$u_{t+1} = \frac{u_t + \bar u_t}{\left\|u_t + \bar u_t\right\|}$. The proof is divided into three steps, where we show that the $\ell_2$ error of $\tilde u_t$, $\bar u_t$ and $u_{t+1}$ from their noiseless counterparts, denoted $\tilde u_t^{(0)}$, $\bar u_t^{(0)}$ and $u_{t+1}^{(0)}$ respectively, is bounded by terms linear in $\eta$.

Let us begin with the first step, in which we show that $\left\|\tilde u_t - \tilde u_t^{(0)}\right\| \lesssim \eta$. 
Plugging $X = \sigma uv^\top + Z$ into Eq.~\eqref{eq:minNormSol_generalX}, 
the column space part of the minimal-norm solution 
of \eqref{eq:ILS_update_rule_step1} is
\begin{align}
\tilde{u}_t &= \beta'_t u + Z v_t - \frac{\alpha_t \beta'_t + u_t^\top Z v_t}{2} u_t 
= \tilde u_t^{(0)} + \eta \left(I - \frac 12 u_t u_t^\top\right) \bar z_{v,t}, \label{eq:utilde_withZ}
\end{align}
where $\beta'_t = \beta_t\sigma$, $\bar z_{v,t} = \bar Z v_t$ and $\tilde u_t^{(0)} = \beta'_t\left(u - \frac{\alpha_t}{2} u_t\right)$ .
For future use, note that
\begin{align}\label{eq:ut_utilde0}
u_t^\top \tilde u_t^{(0)} = \frac{\alpha_t \beta'_t}{2}
\quad \mbox{and} \quad
\left\| \tilde u_t^{(0)} \right\| = \tilde\alpha_t \beta'_t
\geq \frac{\beta'_t}{2}
\end{align}
where $\tilde\alpha_t = \sqrt{1 - \frac{3\alpha_t^2}{4}}$.
Next, 
let $\zeta_{t,1} \equiv{\bar z_{v,t}}^\top u$ and $\zeta_{t,2} \equiv {\bar z_{v,t}}^\top u_t$. Since $\|u\| = \|u_t\| = \|v_t\| = 1$,
\begin{align}\label{eq:sigma1_bound}
|\zeta_{t,1}|, |\zeta_{t,2}|, \|\bar z_{v,t}\| \leq \sigma_1\left(\bar Z\right).
\end{align}
By Lemma~\ref{lem:Zbar_maximalSV}, with probability  at least $1 - e^{-\frac{n}{2}}$, $\sigma_1\left(\bar Z\right) \leq 2 + \rho \leq 3$. From now on, we assume this event holds.

To bound $\|\tilde u_t - \tilde u_t^{(0)}\|$, it is convenient to decompose $\tilde u_t$, given by Eq.~\eqref{eq:utilde_withZ}, into its component in the direction $\tilde u_t^{(0)}$ and an orthogonal component $\eta \tilde w$ with $\tilde u_t^{(0)} \perp \tilde w$:
\begin{align}\label{eq:utilde_decomposition}
\tilde u_t = (1 + \eta\tilde r) \tilde u_t^{(0)} + \eta \tilde w
\end{align}
where, by Eq.~\eqref{eq:ut_utilde0},  
\begin{align}\label{eq:rtilde}
\tilde r &= {\bar z_{v,t}}^\top \left(I - \tfrac 12 u_t u_t^\top \right) \frac{\tilde u_t^{(0)}}{\left\|\tilde u_t^{(0)}\right\|^2}
= \frac{1}{\tilde\alpha_t^2 \beta'_t} \left(\zeta_{t,1} - \frac 34 \alpha_t \zeta_{t,2}\right)
\end{align}
and
\begin{align*}
\tilde w &= \left( I - \tfrac 12 u_t u_t^\top \right) \bar z_{v,t} - \tilde r \tilde u_t^{(0)} = \bar z_{v,t} - \frac 12 \zeta_{t,2} u_t - \tilde r \tilde u_t^{(0)}.
\end{align*}
For future use, and to conclude the first step of the proof, we now bound $|\tilde r|$ and $\|\tilde w\|$. Denote $D_1 \equiv \frac{1}{\delta\sigma}$. Combining Eq.~\eqref{eq:ut_utilde0}, \eqref{eq:sigma1_bound}, $\sigma_1\left(\bar Z\right)\leq 3$ and the assumption $\alpha_t, \beta_t \geq \delta > 0$ gives 
\begin{align}\label{eq:rtilde_magnitude}
|\tilde r| \leq \frac{1+\frac 34 \alpha_t}{\tilde\alpha_t^2 \beta'_t} \sigma_1\left(\bar Z\right) \leq 21D_1.
\end{align}
Combining Eqs.~\eqref{eq:ut_utilde0}, \eqref{eq:rtilde}, \eqref{eq:sigma1_bound} and $\sigma_1\left(\bar Z\right) \leq 3$ 
and some algebraic manipulations yield
\begin{align}
\left\| \tilde w \right\|^2
&= \| \bar z_{v,t} \|^2 - \frac{(4\zeta_{t,1} - 3\alpha_t\zeta_{t,2})^2}{16\tilde \alpha_t^2} - \frac 34 \zeta_{t,2}^2
\leq \| \bar z_{v,t} \|^2
\leq \sigma_1\left(\bar Z\right)^2
\leq 9. \label{eq:omegatilde_norm}
\end{align}
Since the bound on $\|\tilde w\|$ is independent of $\eta$, \eqref{eq:utilde_decomposition} concludes the first step of the proof.

Next, we bound $\left\|\bar u_t - \bar u_t^{(0)}\right\|$.
By \eqref{eq:utilde_decomposition}, $\left\|\tilde u_t\right\|^2 = (1 + \tilde r\eta)^2 \left\|\tilde u_t^{(0)}\right\|^2 + \eta^2 \|\tilde w\|^2$.
Thus
\begin{align*}
\bar u_t &= \frac{(1 + \tilde r\eta) \tilde u_t^{(0)} + \eta \tilde w}{\sqrt{(1 + \tilde r\eta)^2 \left\|\tilde u_t^{(0)}\right\|^2 + \eta^2 \|\tilde w\|^2}}
= \frac{\bar u_t^{(0)}}{\sqrt{1 + \tilde \gamma^2 \tilde\eta^2}} + \frac{\tilde\eta \tilde \gamma }{\sqrt{1 + \tilde\gamma^2 \tilde\eta^2 } } \cdot 
\frac{\tilde w}{\left\|\tilde w\right\|}
\end{align*}
where $\tilde\eta = \frac{\eta}{\left|1 + \tilde r\eta\right|}$ and $\tilde \gamma = \frac{\|\tilde w\|}{\left\|\tilde u_t^{(0)}\right\|}$.
Combining Eqs. \eqref{eq:omegatilde_norm} and~\eqref{eq:ut_utilde0} and $\beta_t \geq \delta$ gives that $\tilde\gamma \leq 6D_1$.
This, together with the triangle inequality,
implies that 
\begin{align}
\left\| \bar u_t - \bar u_t^{(0)} \right\|^2 &= 
\left\| \tilde a_t \bar u_t^{(0)}
-
\frac{\tilde\eta \tilde \gamma }{\sqrt{1 + \tilde\gamma^2 \tilde\eta^2 } } \cdot 
\frac{\tilde w}{\left\|\tilde w\right\|}
\right\|^2
\leq \left(\tilde a_t \left\| \bar u_t^{(0)} \right\| + \tilde\eta \tilde \gamma \right)^2 \leq 
	\left(\tilde a_t + 6D_1\tilde\eta\right)^2 \label{eq:ubar_ubar0_distance}
\end{align}
where $\tilde a_t = 1 - \frac{1}{\sqrt{1 + \tilde \gamma^2 \tilde\eta^2}}$.
Next, 
we bound 
$\tilde a_t$ and $\tilde\eta$
in Eq.~\eqref{eq:ubar_ubar0_distance}.
By assumption \eqref{eq:lowNoiseLevel_condition}, 
$D_1\eta \leq \frac{1}{75}$.
Combining this with Eq.~\eqref{eq:rtilde_magnitude} yields
$|\tilde r| \eta \leq \frac{1}{3}$, and thus 
\begin{align}\label{eq:etatilde_magnitude}
\tilde\eta \leq \frac{3}{2}\eta.
\end{align}
To bound $\tilde a_t$, we use $\forall x:\, 1-\frac{1}{\sqrt{1+x^2}} \leq \frac{x^2}{2}$, which implies $\tilde a_t \leq \frac 12 \tilde \gamma^2 \tilde\eta^2 \leq \frac 98 \tilde\gamma^2 \eta^2$.
Since $\tilde \gamma < 6 D_1$, 
\begin{align}\label{eq:atilde_magnitude}
\tilde a_t \leq 41 D_1^2 \eta^2.
\end{align}
Inserting Eqs.~\eqref{eq:etatilde_magnitude} and \eqref{eq:atilde_magnitude} into Eq.~\eqref{eq:ubar_ubar0_distance} and recalling that $D_1\eta \leq \frac{1}{75}$ yields
\begin{align*}
\left\| \bar u_t - \bar u_t^{(0)} \right\|^2
&\leq \left(41D_1^2\eta^2 + 9D_1\eta \right)^2 \leq\left(10D_1\right)^2 \eta^2 = D_2^2 \eta^2,
\end{align*}
where $D_2 \equiv 10D_1$. Since $D_2$ is independent of $\eta$, the second step follows.

We now prove the third and final step.
To this end, we decompose $(u_t + \bar u_t)$ into its component in the direction  $u_t + \bar u_t^{(0)}$ and an orthogonal component $\eta \bar w$ with $\left(u_t + \bar u_t^{(0)}\right) \perp \bar w$:
\begin{align}\label{eq:ubar_decomposition}
u_t + \bar u_t = (1 + \bar r \eta) \left(u_t + \bar u_t^{(0)}\right) + \eta \bar w
\end{align}
where $\bar r = \left(\bar w'\right)^\top \frac{u_t + \bar u_t^{(0)}}{\left\| u_t + \bar u_t^{(0)} \right\|^2}$ 
and $\bar w = \bar w' - \bar r \left(u_t + \bar u_t^{(0)}\right)$.
Next, we bound $|\bar r|$ and $\left\| \bar\omega_t\right\|$.
Combining Eq.~\eqref{eq:ut_utilde0} and Lemma~\ref{lem:noisyAlphaBeta_geq_delta} gives that $\left\| u_t + \bar u_t^{(0)} \right\|^2 = \left\| u_t + \frac{u - \frac 12\alpha_t u_t}{\tilde \alpha_t} \right\|^2 = 2 + \frac{2\alpha_t}{\tilde\alpha_t} \geq 2$. Recall that $\|u_t\| = \left\|\bar u_t^{(0)} \right\| = 1$ and $\left\|\bar w'\right\| \leq D_2$. Hence, 
\begin{align}\label{eq:rbar_magnitude}
|\bar r| \leq \frac{\left\| \bar w'\right\| \left(\|u_t\| + \left\|\bar u_t^{(0)}\right\|\right)}{\left\| u_t + \bar u_t^{(0)} \right\|^2} \leq D_2 .
\end{align}
This, in turn, together with the triangle inequality, implies the second bound
\begin{align}\label{eq:omegabar_norm}
\|\bar w\|^2 &\leq \left(\left\|\bar w'\right\| + |\bar r| \left\|u_t + \bar u_t^{(0)}\right\|\right)^2 
\leq \left(\left\|\bar w'\right\| + |\bar r| \left(\|u_t\| + \left\|\bar u_t^{(0)}\right\|\right)\right)^2 \leq 9D_2^2.
\end{align}
Following \eqref{eq:ubar_decomposition}, $\| u_t + \bar u_t\|^2 = \left(1 + \bar r \eta\right)^2 \left\| u_t + \bar u_t^{(0)} \right\|^2 + \eta^2 \|\bar w\|^2$,
so the next estimate is
\begin{align*}
u_{t+1} &= \frac{(1 + \bar r \eta) \left(u_t + \bar u_t^{(0)}\right) + \eta \bar w}{\sqrt{\left(1 + \bar r \eta\right)^2 \left\| u_t + \bar u_t^{(0)}\right\|^2 + \eta^2 \|\bar w\|^2}}
= \frac{u_{t+1}^{(0)}}{\sqrt{1 + \bar\gamma^2 \bar\eta^2}} + \frac{\bar\eta \bar\gamma}{\sqrt{1 + \bar\gamma^2 \bar\eta^2}}\cdot \frac{\bar w}{\left\|\bar w\right\|}
\end{align*}
where $\bar \eta = \frac{\eta}{|1 + \bar r\eta|}$ and 
$\bar\gamma = \frac{\left\|\bar w\right\|}{\left\|u_t + \bar u_t^{(0)}\right\|}$.
Combining $\frac{1}{\left\|u_t + \bar u_t^{(0)}\right\|} \leq \frac{1}{\sqrt 2}$ and Eq.~\eqref{eq:omegabar_norm} gives that $\bar\gamma^2 \leq 5D_2^2$.
This, together with the triangle inequality,
implies that 
\begin{align}
\left\| u_{t+1} - u_{t+1}^{(0)} \right\|^2 &= \left\| 
\bar a u_{t+1}^{(0)} - \frac{\bar\eta \bar\gamma}{\sqrt{1 + \bar\gamma^2 \bar\eta^2}}\cdot \frac{\bar w}{\left\|\bar w\right\|} \right\|^2 \leq \left(\bar a \left\| u_{t+1}^{(0)} \right\| + \bar\eta \bar\gamma \right)^2 \leq \left(\bar a + \sqrt 5D_2\bar\eta\right)^2 \label{eq:utp1_utp10_distance}
\end{align}
where $\bar a = 1 - \frac{1}{\sqrt{1 + \bar \gamma^2 \bar\eta^2}}$.
To show that $\left\| \bar u_{t+1} - u_{t+1}^{(0)} \right\|$ is bounded by a term linear in $\eta$, we bound the two quantities 
$\bar a$ and $\bar\eta$
in Eq.~\eqref{eq:utp1_utp10_distance}.
Recall that $D_2 = 10D_1$.
Since $D_1 \eta \leq \frac{1}{75}$, we have $D_2\eta \leq \frac{1}{7}$. Combining it with Eq.~\eqref{eq:rbar_magnitude} yields $|\bar r|\eta \leq \frac{1}{7}$, which implies the first bound
\begin{align}\label{eq:etabar_magnitude}
\bar\eta \leq \tfrac{7}{6}\eta.
\end{align}
To bound $\bar a$, we again use $\forall x:\, 1-\frac{1}{\sqrt{1+x^2}} \leq \frac{x^2}{2}$ which implies $\bar a \leq \frac 12 \bar \gamma^2 \bar\eta^2 \leq \frac 45 \bar\gamma^2 \eta^2$. Since $\gamma^2 \leq 5D_2^2$, we obtain
$\bar a \leq 4 D_2^2\eta^2 .$
Combining this bound on $\bar a$ with Eqs.~\eqref{eq:etabar_magnitude}, \eqref{eq:utp1_utp10_distance} and recalling that $D_2\eta \leq \frac{1}{7}$ and $D_2 = 10D_1$ yields
\begin{align*}
\left\| u_{t+1} - u_{t+1}^{(0)} \right\|^2
&\leq \left(4D_2^2\eta^2 + \tfrac{7\sqrt 5}{6}D_2\eta\right)^2 \leq \left(3.2\cdot 10D_1\eta\right)^2
\leq \left(\frac{50}{\sqrt 2\delta} \frac{\eta}{\sigma}\right)^2.
\end{align*}
\end{proof}


\section{Proof of Auxiliary Lemma \ref{lem:minNormSol_generalX}}\label{sec:appendix_proofOfMinNormSolLemma}

For simplicity, we denote $\bar{u}=u_t$, $\bar{v}=v_t$ and $\Vectorize(X) = \Vectorize_{[m]\times [n]}(X)$ for the vector with entries $X_{i,j}$ for $(i,j) \in [m]\times [n]$. 
Let $\tilde{A}$ be the matrix corresponding to the linear operator
\begin{equation*}
\tilde{A}
\begin{pmatrix} a\\ b \end{pmatrix} = \Vectorize \left( \bar{u} b^{\top} + a\bar{v}^{\top}\right)
\end{equation*}
Note that in the rank-1 case, the least squares problem of Eq.~\eqref{eq:ILS_update_rule_step1} can be rewritten as
\begin{align}\label{eq:leastSquares_usingAtilde}
\text{argmin}_{a\in \mathbb{R}^m, b\in \mathbb{R}^n} \|\tilde{A} \begin{pmatrix} a \\ b \end{pmatrix} - \text{Vec}(X)\|_F.
\end{align}
Next, denote $f_n(i) := \lceil \frac{i}{n} \rceil$ and $g_n(i) := i \text{ mod } n $. 
Then, $\left[\Vectorize(B)\right]_i = B_{f_n(i),g_n(i)}$. Hence
\begin{align*}
\tilde{A}_{ij} &= \begin{cases} \bar{v}_{g_n(i)} \delta_{f_n(i),j}, & j \leq m, \\
\bar{u}_{f_n(i)} \delta_{g_n(i),j-m}, & j > m. \end{cases} 
\end{align*}
Finally, the pseudo-inverse of the matrix $\tilde A$ in (\ref{eq:leastSquares_usingAtilde}) is given by the following lemma:

\begin{lemma}\label{lem:pseudoinverse}
The Moore-Penrose pseudoinverse of $\tilde{A}$, $\tilde{A}^\dagger \in \mathbb{R}^{(m+n) \times (m\cdot n)}$, is given by
\begin{align}
\tilde{A}^\dagger_{ij} &= \begin{cases} \frac{\bar{v}_{g_n(j)}}{\|\bar{v}\|^2} \left(\delta_{i, f_n(j)} - \frac 1N \bar{u}_i \bar{u}_{f_n(j)} \right), & i \leq m, \\
\frac{\bar{u}_{f_n(j)}}{\|\bar{u}\|^2} \left(\delta_{i-m, g_n(j)} - \frac 1N \bar{v}_{i-m} \bar{v}_{g_n(j)} \right), & i > m \end{cases}
	\label{eq:A_dagger}
\end{align}
where $N = \|\bar{u}\|^2 + \|\bar{v}\|^2$.
\end{lemma}

\begin{proof}
We need to show that (i) $\tilde{A}\tilde{A}^\dagger\tilde{A} = \tilde{A}$, (ii) $\tilde{A}^\dagger\tilde{A}\tilde{A} ^\dagger= \tilde{A}^\dagger$, and that (iii) $\tilde{A}\tilde{A}^\dagger$ and (iv) $\tilde{A}^\dagger\tilde{A}$ are Hermitian. 
By relatively simple calculations, 
the entries of $\tilde{A}\tilde{A}^\dagger \in \mathbb{R}^{(mn)\times (mn)}$ are
\begin{align*}
\left(\tilde{A}\tilde{A}^\dagger\right)_{ij} =
\delta_{f_n(i),f_n(j)}\delta_{g_n(i),g_n(j)} - \left(\delta_{f_n(i),f_n(j)} - \frac{\bar{u}_{f_n(i)} \bar{u}_{f_n(j)}}{||\bar u||^2}\right) \left(\delta_{g_n(i),g_n(j)} - \frac{\bar{v}_{g_n(i)} \bar{v}_{g_n(j)}}{||\bar v||^2}\right).
\end{align*}
Similar calculations for $\tilde{A}^\dagger \tilde{A} \in \mathbb{R}^{(m+n)\times(m+n)}$ give 
\begin{align*}
\left(\tilde{A}^\dagger\tilde{A}\right)_{ij} = \begin{cases}
\delta_{i,j} - \frac{1}{N} \bar{u}_i\bar{u}_j,  & i\leq m \text{ and } j\leq m, \\
\frac 1N \bar{u}_i \bar{v}_{j-m}, & i\leq m \text{ and } j>m, \\
\frac 1N \bar{v}_{i-m} \bar{u}_j, & i>m \text{ and } j\leq m, \\
\delta_{i,j} - \frac{1}{N} \bar{v}_{i-m}\bar{v}_{j-m}, & i>m \text{ and } j>m.
\end{cases}
\end{align*}
Since these two matrices are Hermitian, conditions (iii)-(iv) are fulfilled. 
It is now simple to verify that 
$\tilde A\tilde A^\dagger \tilde A= \tilde A$ and $\tilde A^\dagger \tilde A \tilde A^\dagger = \tilde A^\dagger$. 
%
Thus conditions (i)-(ii) are also fulfilled.
\end{proof}

\begin{proof}[Proof of Lemma \ref{lem:minNormSol_generalX}]
Since problem \eqref{eq:ILS_update_rule_step1} is equivalent to \eqref{eq:leastSquares_usingAtilde}, its 
minimal norm solution is $\begin{pmatrix} \tilde{u} \\ \tilde{v} \end{pmatrix} = \tilde A^\dagger \mbox{Vec}(X)$,
with $\tilde A^\dagger$ given by \eqref{eq:A_dagger}. 
An explicit calculation yields \eqref{eq:minNormSol_generalX}. Eq.~\eqref{eq:minNormSol_rank1X} follows by plugging $X = \sigma uv^\top$.
\end{proof}


\section*{Acknowledgments}
We would like to especially thank Yuval Kluger for pointing us at the direction of matrix completion and for inspiring conversations along the way. We thank Laura Balzano, Nicolas Boumal, Rachel Ward, Rong Ge, Eric Chi, Chen Greif and Haim Avron for interesting discussions.
BN is incumbent of the 
William Petschek professorial chair of mathematics.
BN was supported by NIH grant R01GM135928 and by Pazy  foundation grant ID77-2018. 
 Part of this work was done while BN was on sabbatical at the Institute for Advanced Study at Princeton. He gratefully acknowledges
 the support from the Charles Simonyi Endowment.  


\bibliography{Arxiv_R2RILS}
\bibliographystyle{plainnat}

\end{document}